\documentclass[
final
, nomarks
]{dmtcs-episciences}

\usepackage{enumitem, amsthm, amssymb, amsmath, amsfonts, mathtools, rotating, subcaption, float, tikz, booktabs, diagbox}
\usetikzlibrary{automata, calc}

\DeclareMathOperator{\Cay}{Cay}
\DeclareMathOperator{\Dih}{Dih}
\DeclareMathOperator{\rad}{rad}
\DeclareMathOperator{\diag}{diag}
\newcommand\Osq{\mathbin{\text{\scalebox{.84}{$\square$}}}}

\theoremstyle{plain}
\newtheorem{theorem}{Theorem}[section]
\newtheorem{lemma}[theorem]{Lemma}
\newtheorem{claim}[theorem]{Claim}
\newtheorem{corollary}[theorem]{Corollary}
\newtheorem{proposition}[theorem]{Proposition}

\theoremstyle{definition}

\newtheorem{problem}[theorem]{Problem}

\usepackage[utf8]{inputenc}

\usepackage[numbers,sort&compress]{natbib}

\author[Ivan Damnjanovi\'c]{Ivan Damnjanovi\'c\thanks{The author is supported by the Ministry of Science, Technological Development and Innovation of the Republic of Serbia, grant number 451-03-137/2025-03/200102, and the Science Fund of the Republic of Serbia, grant \#6767, Lazy walk counts and spectral radius of threshold graphs --- LZWK.}}
\title[Vertex-transitive nut graph order--degree existence problem]{Vertex-transitive nut graph\\ order--degree existence problem}
\affiliation{
  Faculty of Mathematics, Natural Sciences and Information Technologies, University of Primorska, Koper, Slovenia\\
  Faculty of Electronic Engineering, University of Niš, Niš, Serbia}
\begin{document}
\publicationdata{vol. 28:2}{2026}{6}{10.46298/dmtcs.15989}{2025-07-04; 2025-07-04; 2025-11-24}{2025-12-28}
\maketitle
\begin{abstract}
    A nut graph is a nontrivial simple graph whose adjacency matrix has a simple eigenvalue zero such that the corresponding eigenvector has no zero entries. It is known that the order $n$ and degree $d$ of a vertex-transitive nut graph satisfy $4 \mid d$, $d \ge 4$, $2 \mid n$ and $n \ge d + 4$; or $d \equiv 2 \pmod 4$, $d \ge 6$, $4 \mid n$ and $n \ge d + 6$. Here, we prove that for each such $n$ and $d$, there exists a $d$-regular Cayley nut graph of order $n$. As a direct consequence, we obtain all the pairs $(n, d)$ for which there is a $d$-regular vertex-transitive (resp.\ Cayley) nut graph of order $n$.
\end{abstract}

\section{Introduction}

We consider all the graphs to be undirected, finite and simple, and use $V(G)$ to denote the vertex set of a graph $G$. A \emph{nut graph} is a nontrivial graph such that its adjacency matrix has a simple eigenvalue zero with the corresponding eigenvector having no zero entries. The nut graphs were introduced as a mathematical curiosity in the 1990s by Sciriha and Gutman \cite{Sciriha1997, Sciriha1998_A, Sciriha1998_B, Sciriha1999, ScGu1998}, while the chemical justification for studying these graphs was later discovered through a series of papers \cite{ScFo2007, ScFo2008, FoPiToBoSc2014, FoPiBa2021}. An algorithm for generating nonisomorphic nut graphs was subsequently implemented by Coolsaet, Fowler and Goedgebeur \cite{CoFoGo2018}, while the notion of nut graph was generalized to signed graphs \cite{BaFoPiSci2020} and directed graphs \cite{BaFoMcCaPo2025}. For more results on nut graphs, the reader is referred to \cite{Sciriha2007, Sciriha2008} and the monograph \cite{ScFa2021} by Sciriha and Farrugia.

A \emph{vertex-transitive graph} is a graph $G$ whose automorphism group acts transitively on $V(G)$. For any group $\Gamma$ with the identity $e$ and a subset $C \subseteq \Gamma \setminus \{ e \}$ closed under inversion, by $\Cay(\Gamma, C)$ we denote the graph $G$ such that:
\begin{enumerate}[label=\textbf{(\roman*)}]
    \item $V(G) = \Gamma$; and
    \item any two vertices $u, v \in \Gamma$ are adjacent if and only if $v u^{-1} \in C$.
\end{enumerate}
In this context, we refer to $C$ as the corresponding connection set. A \emph{Cayley graph} is a graph that is isomorphic to $\Cay(\Gamma, C)$ for some finite group $\Gamma$ and connection set $C$. A \emph{circulant graph} is a graph that has an automorphism with a single orbit, or equivalently, it is a Cayley graph where the group $\Gamma$ is cyclic.

Here, we consider several realizability problems concerning the existence of $d$-regular nut graphs of order $n$ belonging to a certain class, for given parameters $d$ and $n$. To this end, for any $d \in \mathbb{N}_0$, let $\mathfrak{N}_d^{\mathrm{reg}}$ be the set of all the $n \in \mathbb{N}$ for which there exists a $d$-regular nut graph of order $n$. Similarly, let $\mathfrak{N}_d^{\mathrm{VT}}$ (resp.\ $\mathfrak{N}_d^{\mathrm{Cay}}$, $\mathfrak{N}_d^{\mathrm{circ}}$) be the set of all the orders attainable by a $d$-regular vertex-transitive (resp.\ Cayley, circulant) nut graph. Clearly,
\[
    \mathfrak{N}_d^{\mathrm{circ}} \subseteq \mathfrak{N}_d^{\mathrm{Cay}} \subseteq \mathfrak{N}_d^{\mathrm{VT}} \subseteq \mathfrak{N}_d^{\mathrm{reg}}
\]
holds for each $d \in \mathbb{N}_0$. We also trivially observe that $\mathfrak{N}_0^{\mathrm{reg}} = \mathfrak{N}_1^{\mathrm{reg}} = \mathfrak{N}_2^{\mathrm{reg}} = \varnothing$.

The study of regular nut graphs was initiated by Gauci, Pisanski and Sciriha through the following order--degree existence problem.

\begin{problem}[\hspace{1sp}{\cite[Problem~12]{GaPiSc2023}}]\label{reg_problem}
    For each degree $d$, determine the set $\mathfrak{N}_d^{\mathrm{reg}}$.
\end{problem}

\noindent
In the same paper, the next initial result was obtained.

\begin{theorem}[\hspace{1sp}{\cite[Theorems~2~and~3]{GaPiSc2023}}]\label{reg_base_th_1}
The following holds:
\[
    \mathfrak{N}^{\mathrm{reg}}_3 = \{12\} \cup \{ n \in \mathbb{N} : \mbox{$n$ is even and } n \ge 18 \} \quad \mbox{and} \quad \mathfrak{N}^{\mathrm{reg}}_4 = \{ 8, 10, 12\} \cup  \{ n \in \mathbb{N} : n \ge 14 \} .
\]
\end{theorem}

\noindent
This result was subsequently extended by Fowler, Gauci, Goedgebeur, Pisanski and Sciriha as follows.

\begin{theorem}[\hspace{1sp}{\cite[Theorem~7]{FoGaGoPiSc2020}}]
The following statements hold:
\begin{enumerate}[label=\textbf{(\roman*)}]
    \item $\mathfrak{N}^{\mathrm{reg}}_5 = \{ n \in \mathbb{N} : \mbox{$n$ is even and } n \ge 10 \}$;
    \item $\mathfrak{N}^{\mathrm{reg}}_6 = \{ n \in \mathbb{N} : n \ge 12 \}$;
    \item $\mathfrak{N}^{\mathrm{reg}}_7 = \{ n \in \mathbb{N} : \mbox{$n$ is even and } n \ge 12 \}$;
    \item $\mathfrak{N}^{\mathrm{reg}}_8 = \{ 12 \} \cup \{ n \in \mathbb{N} : n \ge 14 \}$;
    \item $\mathfrak{N}^{\mathrm{reg}}_9 = \{ n \in \mathbb{N} : \mbox{$n$ is even and } n \ge 16 \}$;
    \item $\mathfrak{N}^{\mathrm{reg}}_{10} = \{ n \in \mathbb{N} : n \ge 15 \}$;
    \item $\mathfrak{N}^{\mathrm{reg}}_{11} = \{ n \in \mathbb{N} : \mbox{$n$ is even and } n \ge 16 \}$.
\end{enumerate}
\end{theorem}

\noindent
Later on, the set $\mathfrak{N}_{12}^{\mathrm{reg}}$ was also determined by Bašić, Knor and Škrekovski.

\begin{theorem}[\hspace{1sp}{\cite[Theorem~1.3]{BaKnSk2022}}]
$\mathfrak{N}^{\mathrm{reg}}_{12} = \{ n \in \mathbb{N} : n \ge 16 \}$.
\end{theorem}

Fowler, Gauci, Goedgebeur, Pisanski and Sciriha initiated the vertex-transitive nut graph order--degree existence problem by posing the next question.

\begin{problem}[\hspace{1sp}{\cite[Question~9]{FoGaGoPiSc2020}}]\label{vt_problem}
    For what pairs $(n, d)$ does a vertex-transitive nut graph of order $n$ and degree~$d$ exist?
\end{problem}

\noindent
In the same paper, the following necessary condition for Problem \ref{vt_problem} was proved.

\begin{theorem}[\hspace{1sp}{\cite[Theorem~10]{FoGaGoPiSc2020}}]\label{base_vt_th}
Let $G$ be a vertex-transitive nut graph on $n$ vertices, of degree $d$. Then $n$ and $d$ satisfy the following conditions. Either $d \equiv 0 \pmod 4$, and $n \equiv 0 \pmod 2$ and $n \ge d + 4$; or $d \equiv 2 \pmod 4$, and $n \equiv 0 \pmod 4$ and $n \ge d + 6$.    
\end{theorem}

\noindent
The circulant and Cayley nut graphs, which both form a subclass of the vertex-transitive nut graphs, were then investigated through a series of papers \cite{DaSt2022, Damnjanovic2023_FIL, Damnjanovic2024_AMC, Damnjanovic2024_DMC, Damnjanovic2025_ADAM}, leading to the following two results.

\begin{theorem}[\hspace{1sp}{\cite[Theorem 1.8]{Damnjanovic2024_AMC}}]\label{base_circ_th}
For each $d \in \mathbb{N}_0$, the set $\mathfrak{N}^{\mathrm{circ}}_d$ is given by
\[
    \mathfrak{N}^{\mathrm{circ}}_d = \begin{cases}
        \varnothing,& \mbox{if $d = 0$ or $4 \nmid d$},\\
        \{ n \in \mathbb{N} : \mbox{$n$ is even and } n \ge d + 4 \},& \mbox{if $d \equiv 4 \pmod 8$},\\
        \{14\} \cup \{ n \in \mathbb{N} : \mbox{$n$ is even and } n \ge 18 \},& \mbox{if $d = 8$},\\
        \{ n \in \mathbb{N} : \mbox{$n$ is even and } n \ge d + 6 \},& \mbox{if $8 \mid d$ and $d \ge 16$}.
    \end{cases}
\]
\end{theorem}

\begin{theorem}[\hspace{1sp}{\cite[Corollaries 8 and 9]{Damnjanovic2025_ADAM}}]\label{base_cay_th}
    For each $d \in \mathbb{N}$ such that $4 \mid d$, the sets $\mathfrak{N}^{\mathrm{VT}}_d$ and $\mathfrak{N}^{\mathrm{Cay}}_d$ are given by
    \[
        \mathfrak{N}^{\mathrm{VT}}_d = \mathfrak{N}^{\mathrm{Cay}}_d = \{ n \in \mathbb{N} : \mbox{$n$ is even and } n \ge d + 4 \} .
    \]
\end{theorem}

\noindent
The closely related polycirculant nut graphs were studied in \cite{BaDam2025_CAM, DaBaPiZi2024, DaBaPiZi2025}. For other recent results concerning the automorphisms of nut graphs, the reader is referred to \cite{BaDam2025_JACO, BaFo2024, BaFoPi2024}.

The following result on the degrees of regular and Cayley nut graphs was recently obtained.

\begin{theorem}[\hspace{1sp}{\cite{BaDamFo2025}}]\label{base_reg_cay_th}
    The set $\mathfrak{N}_d^{\mathrm{reg}}$ is infinite for any $d \ge 3$, while the set $\mathfrak{N}_d^{\mathrm{Cay}}$ is infinite for any even $d \ge 4$.
\end{theorem}

\noindent
Here, we completely solve Problem \ref{vt_problem} through a constructive approach by using Cayley nut graphs, thereby extending Theorems \ref{base_cay_th} and \ref{base_reg_cay_th}. Additionally, this gives an inverse result to Theorem \ref{base_vt_th} in the sense of showing that the necessary condition is sharp. Our main result is embodied in the following theorem.

\begin{theorem}\label{main_th}
    For each $d \in \mathbb{N}_0$, the sets $\mathfrak{N}_d^{\mathrm{VT}}$ and $\mathfrak{N}_d^{\mathrm{Cay}}$ are given by
    \[
        \mathfrak{N}_d^{\mathrm{VT}} = \mathfrak{N}_d^{\mathrm{Cay}} = \begin{cases}
            \varnothing,& \mbox{if $d$ is odd or $d < 4$},\\
            \{ n \in \mathbb{N} : \mbox{$n$ is even and } n \ge d + 4 \},& \mbox{if $4 \mid d$ and $d \ge 4$},\\
            \{ n \in \mathbb{N} : \mbox{$4 \mid n$ and } n \ge d + 6 \},& \mbox{if $d \equiv 2 \pmod 4$ and $d \ge 6$}.
        \end{cases}
    \]
\end{theorem}

\noindent
As it turns out, the necessary condition from Theorem \ref{base_vt_th} for the existence of a $d$-regular vertex-transitive nut graph of order $n$ is also sufficient, apart from the trivial case when $d = 0$ or $d = 2$.

In the rest of the paper, our main focus is to prove Theorem \ref{main_th}. In Section \ref{sc_prel}, we overview the theory necessary to carry out the proof. Afterwards, in Section \ref{sc_aux}, we obtain several results on the divisibility of four auxiliary families of polynomials by the cyclotomic polynomials. Finally, in Section \ref{sc_main}, we rely on constructions of Cayley nut graphs based on dihedral groups to complete the proof of Theorem~\ref{main_th} and end the paper with a brief conclusion in Section \ref{sc_conclusion}. The proof of several results from Section \ref{sc_aux} is completed through a computer-assisted approach by using the \texttt{Python} and \texttt{SageMath} \cite{SageMath} scripts that can be found in \cite{GitHub}.

\section{Preliminaries}\label{sc_prel}

For any graph $G$, let $A(G)$ denote the adjacency matrix of $G$ and let $\sigma(G)$ be the spectrum of $A(G)$, regarded as a multiset. Also, let $\overline{G}$ denote the complement of a graph $G$. We will need the following well-known result; for the proof, see the standard literature on spectral graph theory \cite{BrouHae2012, Chung1997, CvetDoobSachs1995, CvetRowSi1997, CvetRowSi2010}.

\begin{lemma}\label{reg_comp_lemma}
    Let $G$ be a regular graph of order $n$ with $\sigma(G) = \{ \lambda_1, \lambda_2, \ldots, \lambda_n \}$, where $\lambda_1 \ge \lambda_2 \ge \cdots \ge \lambda_n$. Then
    \[
        \sigma\left( \overline{G} \right) = \{ n - 1 - \lambda_1, -1 - \lambda_n, -1 - \lambda_{n - 1}, \ldots, -1 - \lambda_2 \} .
    \]
\end{lemma}

\noindent
Given a graph $G$, let $\eta(G)$ denote the multiplicity of zero as an eigenvalue of $A(G)$. The following property of vertex-transitive graphs is well known and follows directly from \cite[p.\ 135]{CvetDoobSachs1995}.

\begin{lemma}\label{vt_nut_lemma}
    Let $G$ be a vertex-transitive graph. Then $\eta(G) = 1$ if and only if $G$ is a nut graph.
\end{lemma}

For each $n \ge 3$, we use $\Dih(n)$ to denote the dihedral group of order $2n$, i.e.,
\[
    \Dih(n) = \langle r, s \mid r^n = s^2 = e, s r s = r^{-1} \rangle .
\]
Here, $e$, $r$ and $s$ signify the identity, the rotation by $\frac{2 \pi}{n}$ and a reflection symmetry, respectively. Besides, for any $n \in \mathbb{N}_0$, we denote the identity matrix of order $n$ by $I_n$, and for any $m, n \in \mathbb{N}_0$, we denote the zero matrix with $m$ rows and $n$ columns by $O_{m, n}$. When the matrix size is clear from context, we may drop the subscripts and write $I$ or $O$ for short.
We resume with the next lemma.

\begin{lemma}\label{bicirc_lemma}
    For some $n \in \mathbb{N}$ and each $j = 0, 1, 2, 3$, let $A^{(j)}$ be the circulant matrix
    \[
        A^{(j)} = \begin{bmatrix}
            a_0^{(j)} & a_1^{(j)} & a_2^{(j)} & \cdots & a_{n-1}^{(j)}\\
            a_{n-1}^{(j)} & a_0^{(j)} & a_1^{(j)} & \cdots & a_{n-2}^{(j)}\\
            a_{n-2}^{(j)} & a_{n-1}^{(j)} & a_0^{(j)} & \cdots & a_{n-3}^{(j)}\\
            \vdots & \vdots & \vdots & \ddots & \vdots\\
            a_1^{(j)} & a_2^{(j)} & a_3^{(j)} & \dots & a_0^{(j)}
        \end{bmatrix} .
    \]
    Then the matrix given in the block form
    \begin{equation}\label{paux_1}
        \begin{bmatrix}
            A^{(0)} & A^{(1)}\\
            A^{(2)} & A^{(3)}
        \end{bmatrix}
    \end{equation}
    is similar to the direct sum
    \[
        \bigoplus_{\zeta} \begin{bmatrix}
            P_0(\zeta) & P_1(\zeta)\\
            P_2(\zeta) & P_3(\zeta)
        \end{bmatrix},
    \]
    where
    \[
        P_j(x) = a_0^{(j)} + a_1^{(j)} x + a_2^{(j)} x^2 + \cdots + a_{n-1}^{(j)} x^{n-1} \quad (j = 0, 1, 2, 3),
    \]
    and $\zeta$ ranges over the $n$-th roots of unity.
\end{lemma}
\begin{proof}
    Let $\omega = e^{2 \pi i / n}$ and let $U \in \mathbb{C}^{n \times n}$ be defined as
    \[
        U_{k, \ell} = \omega^{(k - 1)(\ell - 1)} \quad (k, \ell = 1, 2, \ldots, n) .
    \]
    Observe that $U U^* = U^* U = n I_n$ and $A^{(j)} U = U D^{(j)}$, where
    \[
        D^{(j)} = \diag(P_j(1), P_j(\omega), P_j(\omega^2), \ldots, P_j(\omega^{n - 1})) \quad (j = 0, 1, 2, 3).
    \]
    Therefore,
    \[
        \begin{bmatrix}
            A^{(0)} & A^{(1)}\\
            A^{(2)} & A^{(3)}
        \end{bmatrix} \begin{bmatrix}
        U & O\\
        O & U
        \end{bmatrix} = \begin{bmatrix}
            A^{(0)} U & A^{(1)} U\\
            A^{(2)} U & A^{(3)} U
        \end{bmatrix} = \begin{bmatrix}
            U D^{(0)} & U D^{(1)}\\
            U D^{(2)} & U D^{(3)}
        \end{bmatrix} = \begin{bmatrix}
        U & O\\
        O & U
        \end{bmatrix} \begin{bmatrix}
            D^{(0)} & D^{(1)}\\
            D^{(2)} & D^{(3)}
        \end{bmatrix} ,
    \]
    which implies that the matrix \eqref{paux_1} is similar to
    \begin{equation}\label{paux_2}
        \begin{bmatrix}
            P_0(1) & 0 & \cdots & 0 & P_1(1) & 0 & \cdots & 0\\
            0 & P_0(\omega) & \cdots & 0 & 0 & P_1(\omega) & \cdots & 0\\
            \vdots & \vdots & \ddots & \vdots & \vdots & \vdots & \ddots & \vdots\\
            0 & 0 & \cdots & P_0(\omega^{n - 1}) & 0 & 0 & \cdots & P_1(\omega^{n - 1})\\
            P_2(1) & 0 & \cdots & 0 & P_3(1) & 0 & \cdots & 0\\
            0 & P_2(\omega) & \cdots & 0 & 0 & P_3(\omega) & \cdots & 0\\
            \vdots & \vdots & \ddots & \vdots & \vdots & \vdots & \ddots & \vdots\\
            0 & 0 & \cdots & P_2(\omega^{n - 1}) & 0 & 0 & \cdots & P_3(\omega^{n - 1})
        \end{bmatrix} .
    \end{equation}
    The result now follows by simultaneously rearranging the rows and columns of \eqref{paux_2} in the order $1, n + 1, \linebreak 2, n + 2, 3, n + 3, \ldots, n, 2n$.
\end{proof}

The \emph{connection set} of a binary circulant matrix $C \in \mathbb{R}^{n \times n}$ is the set comprising the integers $j \in \mathbb{Z}_n$ such that $C_{1, 1 + j} = 1$, with the index addition being done modulo $n$. As a direct consequence to Lemmas~\ref{vt_nut_lemma} and~\ref{bicirc_lemma}, we obtain the following result on the nut property of Cayley graphs for the dihedral group.

\begin{lemma}\label{dih_cay_lemma}
    For some $n \ge 3$, let $G$ be the graph $\Cay(\Dih(n), \{ r^{a_1}, r^{a_2}, \ldots, r^{a_k}, r^{b_1} s, r^{b_2} s, \ldots, r^{b_\ell} s \})$, where $k, \ell \in \mathbb{N}_0$, $1 \le a_1 < a_2 < \cdots < a_k < n$ and $0 \le b_1 < b_2 < \cdots < b_\ell < n$. Also, for each $n$-th root of unity $\zeta$, let
    \[
        A_\zeta = \begin{bmatrix}
            \sum_{j = 1}^{k} \zeta^{a_j} & \sum_{j = 1}^{\ell} \zeta^{-b_j}\\
            \sum_{j = 1}^{\ell} \zeta^{b_j} & \sum_{j = 1}^{k} \zeta^{a_j}
        \end{bmatrix} .
    \]
    Then $G$ is a nut graph if and only if exactly one of the $A_\zeta$ matrices has a simple eigenvalue zero, while all the others are invertible.
\end{lemma}
\begin{proof}
Observe that if we arrange the vertices of $G$ as $e, r, r^2, \ldots, r^{n - 1}, s, r^{-1} s, r^{-2} s, \ldots, r^{-(n - 1)} s$, then $A(G)$ has the form
\[
    \begin{bmatrix}
        C_0 & C_1\\
        C_2 & C_0
    \end{bmatrix} ,
\]
where $C_0$, $C_1$ and $C_2$ are the binary circulant matrices with the connection sets
\[
    \{ a_1, a_2, \ldots, a_k \}, \quad \{ -b_1, -b_2, \ldots, -b_\ell \} \quad \mbox{and} \quad \{ b_1, b_2, \ldots, b_\ell \} ,
\]
respectively. Therefore, Lemma~\ref{bicirc_lemma} implies that $A(G)$ is similar to $\bigoplus_{\zeta} A_\zeta$, where $\zeta$ ranges over the $n$-th roots of unity. By Lemma \ref{vt_nut_lemma}, we conclude that $G$ is a nut graph if and only if exactly one of the $A_\zeta$ matrices has a simple eigenvalue zero, while all the others have no eigenvalue zero.
\end{proof}

For any $n \in \mathbb{N}$, the \emph{radical} of $n$, denoted by $\rad(n)$, is the largest square-free positive divisor of $n$. For each $n \in \mathbb{N}$, the \emph{cyclotomic polynomial} $\Phi_n(x)$ is defined as
\[
    \Phi_n(x) = \prod_{\zeta} (x - \zeta),
\]
where $\zeta$ ranges over the primitive $n$-th roots of unity. It is known that for every $n \in \mathbb{N}$, the polynomial $\Phi_n(x)$ has integer coefficients and is irreducible in $\mathbb{Q}[x]$; see, e.g., \cite[Chapter~33]{Gallian2017}. Therefore, any $P(x) \in \mathbb{Q}[x]$ has a root that is a primitive $n$-th root of unity if and only if $\Phi_n(x) \mid P(x)$. The following result is also well known.

\begin{lemma}\label{cyclotomic_reduction_prime}
    Suppose that $p^2 \mid n$, where $n \in \mathbb{N}$ and $p$ is a prime. Then $\Phi_n(x) = \Phi_{n/p}(x^p)$.
\end{lemma}

\noindent
As an immediate consequence of Lemma \ref{cyclotomic_reduction_prime}, we get the next corollary.

\begin{corollary}\label{cyclotomic_reduction_rad}
    For any $n \in \mathbb{N}$, we have $\Phi_n(x) = \Phi_{\rad(n)}(x^{n / \rad(n)})$.
\end{corollary}

\noindent
We will frequently use Corollary \ref{cyclotomic_reduction_rad} together with the following folklore lemma.

\begin{lemma}[\hspace{1sp}{\cite[Lemma 18]{DaBaPiZi2024}}]\label{polynomial_division_lemma}
    Let $V(x), W(x) \in \mathbb{Q}[x], \, W(x) \not\equiv 0$, be such that $W(x) \mid V(x)$ and the powers of all the nonzero terms of $W(x)$ are divisible by $\beta \in \mathbb{N}$. Also, for any $j \in \{0, 1, \ldots, \beta - 1\}$, let $V^{(\beta, j)}(x)$ be the polynomial comprising the terms of $V(x)$ whose power is congruent to $j$ modulo $\beta$. Then $W(x) \mid V^{(\beta, j)}(x)$ for every $j \in \{ 0, 1, \ldots, \beta - 1 \}$.
\end{lemma}

We end the section with the next theorem by Filaseta and Schinzel on the divisibility of lacunary polynomials by cyclotomic polynomials.

\begin{theorem}[\hspace{1sp}{\cite[Theorem~2]{FiSchi2004}}]\label{filaschi_th}
Let $P(x) \in \mathbb{Z}[x]$ have $N$ nonzero terms and suppose that $\Phi_n(x) \mid P(x)$ for some $n \in \mathbb{N}$. Suppose further that $p_1, p_2, \ldots, p_k$ are distinct primes satisfying
\[
    \sum_{j=1}^k (p_j-2) > N - 2.
\]
Let $e_j$ be the largest exponent such that $p_j^{e_j} \mid n$. Then for at least one $j \in \{1, 2, \ldots, k\}$, we have $\Phi_{m}(x)\mid P(x)$, where $m = n / p_j^{e_j}$.
\end{theorem}

\section{Auxiliary polynomials}\label{sc_aux}

In the present section, we investigate the divisibility of four auxiliary families of polynomials by the cyclotomic polynomials. More precisely, we are interested in the polynomials
\begin{align*}
    Q_t(x) \coloneqq \,\, & x^{4t + 7} - x^{4t + 5} - x^{4t + 4} + 2x^{2t + 4}+ x^{2t + 3} + x^{2t + 2} + x^{2t} - 2x^{t + 3} - x^2 - 1,\\
    R_t(x) \coloneqq \,\, &x^{8t + 15} + x^{8t + 14} + x^{8t + 11} - x^{8t + 10} - x^{8t + 8} + 2x^{6t + 9} - x^{4t + 15} - x^{4t + 11}\\
    &- x^{4t + 9} + 2x^{4t + 8} - 2x^{4t + 7} + x^{4t + 6} + x^{4t + 4} + x^{4t} - 2x^{2t + 6} + x^7 + x^5 - x^4 - x - 1,\\
    S_t(x) \coloneqq \,\, & x^{4t + 13} + x^{4t + 11} + x^{4t + 10} + x^{4t + 9} - x^{4t + 8} - x^{2t + 13} - x^{2t + 10} - x^{2t + 9}\\
    &+ 3 x^{2t + 7} + x^{2t + 5} - x^{2t + 4} + x^{2t + 3} - x^{2t + 2} + x^{2t + 1} - 2 x^{t + 6} + x^{6} - x^{5} - x - 1,\\
    T_t(x) \coloneqq \,\, & x^{8t + 27} + x^{8t + 26} + x^{8t + 25} + x^{8t + 22} + x^{8t + 20} + x^{8t + 18} + x^{8t + 17} - x^{8t + 16} - x^{8t + 15}\\
    &+ 2 x^{6t + 15} - x^{4t + 26} - x^{4t + 25} + x^{4t + 23} - x^{4t + 21} - x^{4t + 20} + x^{4t + 19} - x^{4t + 18} - x^{4t + 17}\\
    &+ 3 x^{4t + 14} - 3 x^{4t + 13} + x^{4t + 10} + x^{4t + 9} - x^{4t + 8} + x^{4t + 7} + x^{4t + 6} - x^{4t + 4} + x^{4t + 2}\\
    &+ x^{4t + 1} - 2 x^{2t + 12} + x^{12} + x^{11} - x^{10} - x^9 - x^7 - x^5 - x^2 - x - 1,
\end{align*}
for each $t \in \mathbb{N}_0$. The four subsections of this section correspond to the four families of polynomials that are being studied.

\subsection{\texorpdfstring{$Q_t(x)$}{Qₜ(x)} polynomials}

In this subsection we investigate the $Q_t(x)$ polynomials and our main result is the following lemma.

\begin{lemma}\label{qt_lemma}
    For any $t \in \mathbb{N}_0$, we have $\Phi_b(x) \nmid Q_t(x)$ for each $b \ge 2$.
\end{lemma}

\noindent
We begin with the next claim that can be conveniently proved via computer as shown in \cite{GitHub}.

\begin{claim}\label{power_list_claim_1}
    For each $\beta \ge 6$, there exists an element of the sequence $4t + 7, 4t + 5, 4t + 4, 2t + 4, 2t + 3, \linebreak 2t + 2, 2t, t + 3, 2, 0$ with a unique remainder modulo $\beta$.
\end{claim}

\noindent
We also need the following two auxiliary results.

\begin{claim}\label{radical_claim_1}
    Suppose that for some $t \in \mathbb{N}_0$ and $b \in \mathbb{N}$, we have $\Phi_b(x) \mid Q_t(x)$. Then $\frac{b}{\rad(b)} < 6$.
\end{claim}
\begin{proof}
    By way of contradiction, suppose that $\frac{b}{\rad(b)} \ge 6$. By Corollary \ref{cyclotomic_reduction_rad}, it follows that the powers of all the nonzero terms of $\Phi_b(x)$ are divisible by $\frac{b}{\rad(b)}$. From Lemma \ref{polynomial_division_lemma} and Claim \ref{power_list_claim_1}, we conclude that $\Phi_b(x)$ divides a polynomial of the form $c x^\alpha$ for some $c \in \mathbb{Z} \setminus \{ 0 \}$ and $\alpha \in \mathbb{N}_0$, which yields a contradiction.
\end{proof}

\begin{claim}\label{prime_claim_1}
    For any $t \in \mathbb{N}_0$ and prime $p \ge 11$, we have $\Phi_p(x) \nmid Q_t(x)$.
\end{claim}
\begin{proof}
    By way of contradiction, suppose that $\Phi_p(x) \mid Q_t(x)$. Then $\Phi_p(x)$ also divides the polynomial
    \begin{align*}
        Q_t^{\bmod p}(x) \coloneqq \,\, & x^{(4t + 7) \bmod p} - x^{(4t + 5) \bmod p} - x^{(4t + 4) \bmod p} + 2x^{(2t + 4) \bmod p}\\
        &+ x^{(2t + 3) \bmod p} + x^{(2t + 2) \bmod p} + x^{2t \bmod p} - 2x^{(t + 3) \bmod p} - x^2 - 1 .
    \end{align*}
    Since $\Phi_p(x) = \sum_{j = 0}^{p - 1} x^j$, it follows that $\deg Q_t^{\bmod p}(x) \le p - 1 = \deg \Phi_p(x)$, hence $Q_t^{\bmod p}(x) \equiv 0$ or there is a $c \in \mathbb{Q} \setminus \{ 0 \}$ such that $Q_t^{\bmod p}(x) = c \, \Phi_p(x)$. In the former case, Claim \ref{power_list_claim_1} yields a contradiction. In the latter case, $Q_t^{\bmod p}(x)$ has exactly $p$ nonzero terms, which is impossible because $p \ge 11$.
\end{proof}

\noindent
We are now in a position to complete the proof of Lemma \ref{qt_lemma}.

\begin{proof}[of Lemma \ref{qt_lemma}]
    By way of contradiction, suppose that $\Phi_b(x) \mid Q_t(x)$ holds for some $t \in \mathbb{N}_0$ and $b \ge 2$. By Claim \ref{radical_claim_1}, we have $\frac{b}{\rad(b)} < 6$. If $b$ has no prime factor below $11$, then we can use Theorem~\ref{filaschi_th} to repeatedly cancel out distinct prime factors of $b$ until exactly one is left. Therefore, $\Phi_p(x) \mid Q_t(x)$ holds for some prime $p \ge 11$, which yields a contradiction due to Claim \ref{prime_claim_1}.

    Now, suppose that $b$ has a prime factor below $11$. In this case, Theorem \ref{filaschi_th} can be used to cancel out all the prime factors of $b$ above seven. Hence, $\Phi_{b'}(x) \mid Q_t(x)$ holds for some $b' \ge 2$ whose prime factors belong to $\{2, 3, 5, 7 \}$ and such that $\frac{b'}{\rad(b')} < 6$. Note that there are finitely many such numbers. Besides, by Theorem \ref{filaschi_th}, we can assume without loss of generality that the distinct prime factors $p_1, p_2, \ldots, p_k$ of $b'$ satisfy $\sum_{j = 1}^k (p_j - 2) \le 8$. Observe that $\Phi_{b'}(x) \mid Q_t(x)$ holds if and only if the polynomial
    \begin{align*}
        Q_t^{\bmod b'}(x) \coloneqq \,\, & x^{(4t + 7) \bmod b'} - x^{(4t + 5) \bmod b'} - x^{(4t + 4) \bmod b'} + 2x^{(2t + 4) \bmod b'}\\
        &+ x^{(2t + 3) \bmod b'} + x^{(2t + 2) \bmod b'} + x^{2t \bmod b'} - 2x^{(t + 3) \bmod b'} - x^2 - 1
    \end{align*}
    is divisible by $\Phi_{b'}(x)$. With this in mind, we can obtain a contradiction by going through all the feasible numbers $b'$ and then verifying that $\Phi_{b'}(x) \nmid Q_t^{\bmod b'}(x)$ holds for each $t \in \{ 0, 1, 2, \ldots, b' - 1\}$. This can be done, e.g., via a \texttt{SageMath} script, as shown in \cite{GitHub}.
\end{proof}

\subsection{\texorpdfstring{$R_t(x)$}{Rₜ(x)} polynomials}

Here, we focus on proving the next lemma.

\begin{lemma}\label{rt_lemma}
    For any $t \in \mathbb{N}_0$, we have $\Phi_b(x) \nmid R_t(x)$ for each $b \ge 3$.
\end{lemma}

\noindent
The following result can be proved, e.g., by using a \texttt{Python} script, as shown in \cite{GitHub}.

\begin{claim}\label{power_list_claim_2}
    For each $\beta \ge 11$, there exists an element of the sequence
    \begin{align*}
        8t + 15, 8t &+ 14, 8t + 11, 8t + 10, 8t + 8, 6t + 9, 4t + 15,\\
        4t &+ 11, 4t + 9, 4t + 8, 4t + 7, 4t + 6, 4t + 4, 4t, 2t + 6, 7, 5, 4, 1, 0
    \end{align*}
    with a unique remainder modulo $\beta$.
\end{claim}

\noindent
The next claim can now be proved analogously to Claim \ref{radical_claim_1}.

\begin{claim}\label{radical_claim_2}
    Suppose that for some $t \in \mathbb{N}_0$ and $b \in \mathbb{N}$, we have $\Phi_b(x) \mid R_t(x)$. Then $\frac{b}{\rad(b)} < 11$.
\end{claim}

\noindent
We move to the following two auxiliary claims.

\begin{claim}\label{four_free_claim_2}
    Suppose that for some $t \in \mathbb{N}_0$ and $b \in \mathbb{N}$, we have $\Phi_b(x) \mid R_t(x)$. Then $2^2 \nmid b$.
\end{claim}
\begin{proof}
    By way of contradiction, suppose that $2^2 \mid b$. By Lemma \ref{cyclotomic_reduction_prime}, the powers of all the nonzero terms of $\Phi_b(x)$ are even. From Lemma \ref{polynomial_division_lemma}, we conclude that the polynomial
    \[
        x^{8t + 14} - x^{8t + 10} - x^{8t + 8} + 2x^{4t + 8} + x^{4t + 6} + x^{4t + 4} + x^{4t} - 2x^{2t + 6} - x^4 - 1
    \]
    has a root that is a primitive $b$-th root of unity. Therefore, $Q_t(x)$ has a primitive $\frac{b}{2}$-th root of unity among its roots, which yields a contradiction due to Lemma \ref{qt_lemma}.
\end{proof}

\begin{claim}\label{prime_claim_2}
    For any $t \in \mathbb{N}_0$ and prime $p \ge 23$, we have $\Phi_p(x) \nmid R_t(x)$ and $\Phi_{2p}(x) \nmid R_t(x)$.
\end{claim}
\begin{proof}
    It can be proved analogously to Claim \ref{prime_claim_1} that $\Phi_p(x) \nmid R_t(x)$. Now, by way of contradiction, suppose that $\Phi_{2p}(x) \mid R_t(x)$. Then $\Phi_{2p}(x)$ divides the polynomial
    \begin{align*}
        R_t^{\bmod 2p}(x) \coloneqq \,\, & (-1)^{\lfloor \frac{8t+15}{p} \rfloor} x^{(8t + 15) \bmod p} + (-1)^{\lfloor \frac{8t+14}{p} \rfloor} x^{(8t + 14) \bmod p} + (-1)^{\lfloor \frac{8t+11}{p} \rfloor} x^{(8t + 11) \bmod p}\\
        &- (-1)^{\lfloor \frac{8t+10}{p} \rfloor} x^{(8t + 10) \bmod p} - (-1)^{\lfloor \frac{8t+8}{p} \rfloor} x^{(8t + 8) \bmod p} + 2 (-1)^{\lfloor \frac{6t+9}{p} \rfloor} x^{(6t + 9) \bmod p}\\
        &- (-1)^{\lfloor \frac{4t+15}{p} \rfloor} x^{(4t + 15) \bmod p} - (-1)^{\lfloor \frac{4t+11}{p} \rfloor} x^{(4t + 11) \bmod p} - (-1)^{\lfloor \frac{4t+9}{p} \rfloor} x^{(4t + 9) \bmod p}\\
        &+ 2 (-1)^{\lfloor \frac{4t+8}{p} \rfloor} x^{(4t + 8) \bmod p} - 2 (-1)^{\lfloor \frac{4t+7}{p} \rfloor} x^{(4t + 7) \bmod p} + (-1)^{\lfloor \frac{4t+6}{p} \rfloor} x^{(4t + 6) \bmod p}\\
        &+ (-1)^{\lfloor \frac{4t+4}{p} \rfloor} x^{(4t + 4) \bmod p} + (-1)^{\lfloor \frac{4t}{p} \rfloor} x^{4t \bmod p} - 2 (-1)^{\lfloor \frac{2t+6}{p} \rfloor} x^{(2t + 6) \bmod p}\\
        &+ x^7 + x^5 - x^4 - x - 1 .
    \end{align*}
    Since $\Phi_{2p}(x) = \sum_{j = 0}^{p - 1} (-x)^j$, we have $\deg R_t^{\bmod 2p}(x) \le p - 1 = \deg \Phi_{2p}(x)$, which means that $R_t^{\bmod 2p}(x) \equiv 0$ or there is a $c \in \mathbb{Q} \setminus \{ 0 \}$ such that $R_t^{\bmod 2p}(x) = c \, \Phi_{2p}(x)$. In the former case, Claim~\ref{power_list_claim_2} yields a contradiction, while in the latter case, $R_t^{\bmod 2p}(x)$ has exactly $p$ nonzero terms, which is not possible since $p \ge 23$.
\end{proof}

\noindent
The proof of Lemma \ref{rt_lemma} can now be finalized.

\begin{proof}[of Lemma \ref{rt_lemma}]
    By way of contradiction, suppose that $\Phi_b(x) \mid R_t(x)$ holds for some $t \in \mathbb{N}_0$ and $b \ge 3$. Claims \ref{radical_claim_2} and \ref{four_free_claim_2} imply that $\frac{b}{\rad(b)} < 11$ and $2^2 \nmid b$. If $b$ has no prime factor from $\{3, 5, 7, 11, 13, 17, 19 \}$, then Theorem \ref{filaschi_th} can be used to repeatedly cancel out distinct prime factors of $b$ that are above $19$ until only one such divisor is left. Therefore, $\Phi_p(x) \mid R_t(x)$ or $\Phi_{2p}(x) \mid R_t(x)$ holds for some prime $p \ge 23$, which is impossible due to Claim \ref{prime_claim_2}.

    Now, suppose that $b$ has a prime factor from $\{3, 5, 7, 11, 13, 17, 19 \}$. In this case, Theorem \ref{filaschi_th} can be applied to cancel out all the prime factors of $b$ above $19$, which implies that $\Phi_{b'}(x) \mid R_t(x)$ holds for some $b' \ge 3$ whose prime factors are at most $19$ and such that $\frac{b'}{\rad(b')} < 11$ and $2^2 \nmid b'$. Also, by Theorem~\ref{filaschi_th}, we can assume without loss of generality that the distinct prime factors $p_1, p_2, \ldots, p_k$ of $b'$ satisfy $\sum_{j = 1}^k (p_j - 2) \le 18$. The rest of the proof can be carried out via computer analogously to Lemma~\ref{qt_lemma}; see \cite{GitHub}.
\end{proof}

\subsection{\texorpdfstring{$S_t(x)$}{Sₜ(x)} polynomials}

In the present subsection we study the divisibility of the $S_t(x)$ polynomials by cyclotomic polynomials and obtain the next result.

\begin{lemma}\label{st_lemma}
    For any $t \in \mathbb{N}_0$, we have $\Phi_b(x) \nmid S_t(x)$ for each $b \ge 2$.
\end{lemma}

\noindent
By analogy, we start with the following claim that can be proved via a computer-assisted approach; see \cite{GitHub}.

\begin{claim}\label{power_list_claim_3}
    For each $\beta \ge 8$, there exists an element of the sequence
    \begin{align*}
        4t + 13, 4t &+ 11, 4t + 10, 4t + 9, 4t + 8, 2t + 13, 2t + 10,\\
        2t &+ 9, 2t + 7, 2t + 5, 2t + 4, 2t + 3, 2t + 2, 2t + 1, t + 6, 6, 5, 1, 0
    \end{align*}
    with a unique remainder modulo $\beta$.
\end{claim}

\noindent
The next two results can be proved analogously to Claims \ref{radical_claim_1} and \ref{prime_claim_1}, respectively.

\begin{claim}\label{radical_claim_3}
    Suppose that for some $t \in \mathbb{N}_0$ and $b \in \mathbb{N}$, we have $\Phi_b(x) \mid S_t(x)$. Then $\frac{b}{\rad(b)} < 8$.
\end{claim}

\begin{claim}\label{prime_claim_3}
    For any $t \in \mathbb{N}_0$ and prime $p \ge 23$, we have $\Phi_p(x) \nmid S_t(x)$.
\end{claim}

\noindent
We can now prove Lemma \ref{st_lemma} as follows.

\begin{proof}[of Lemma \ref{st_lemma}]
    By way of contradiction, suppose that $\Phi_b(x) \mid S_t(x)$ holds for some $t \in \mathbb{N}_0$ and $b \ge 2$. From Claim \ref{radical_claim_3}, we obtain $\frac{b}{\rad(b)} < 8$. If $b$ has no prime factor below $23$, then by repeated use of Theorem \ref{filaschi_th}, we conclude that $\Phi_p(x) \mid S_t(x)$ is satisfied for some prime $p \ge 23$. However, by Claim~\ref{prime_claim_3}, this is not possible.

    Now, suppose that $b$ has a prime factor below $23$. By virtue of Theorem \ref{filaschi_th}, we can cancel out all the prime factors of $b$ above $19$. Therefore, $\Phi_{b'}(x) \mid S_t(x)$ holds for some $b' \ge 2$ whose prime factors are at most $19$ and such that $\frac{b'}{\rad(b')} < 8$. By Theorem \ref{filaschi_th}, we can also assume without loss of generality that the distinct prime factors $p_1, p_2, \ldots, p_k$ of $b'$ satisfy $\sum_{j = 1}^k (p_j - 2) \le 17$. Since there are finitely many such numbers $b'$, the proof can be completed analogously to Lemmas \ref{qt_lemma} and \ref{rt_lemma}, e.g., via a \texttt{SageMath} script, as shown in \cite{GitHub}.
\end{proof}

\subsection{\texorpdfstring{$T_t(x)$}{Tₜ(x)} polynomials}

We finish the section with the following lemma concerning the $T_t(x)$ polynomials.

\begin{lemma}\label{tt_lemma}
    For any $t \in \mathbb{N}_0$, we have $\Phi_b(x) \nmid T_t(x)$ for each $b \ge 3$.
\end{lemma}

\noindent
By analogy, we can obtain the next result, so we omit its proof.

\begin{claim}\label{power_list_claim_4}
    For each $\beta \ge 20$, there exists an element of the sequence
    \begin{align}\label{tt_seq}
    \begin{split}
        8t + 27, 8t &+ 26, 8t + 25, 8t + 22, 8t + 20, 8t + 18, 8t + 17, 8t + 16, 8t + 15, 6t + 15, 4t + 26,\\
        4t &+ 25, 4t + 23, 4t + 21, 4t + 20, 4t + 19, 4t + 18, 4t + 17, 4t + 14, 4t + 13, 4t + 10,\\
        4t &+ 9, 4t + 8, 4t + 7, 4t + 6, 4t + 4, 4t + 2, 4t + 1, 2t + 12, 12, 11, 10, 9, 7, 5, 2, 1, 0
    \end{split}
    \end{align}
    with a unique remainder modulo $\beta$.
\end{claim}

\noindent
We resume with the following two claims.

\begin{claim}\label{radical_claim_4}
    Suppose that for some $t \in \mathbb{N}_0$ and $b \in \mathbb{N}$, we have $\Phi_b(x) \mid T_t(x)$. Then $\frac{b}{\rad(b)} < 13$.
\end{claim}
\begin{proof}
    Let $\beta \coloneqq \frac{b}{\rad(b)}$ and by way of contradiction, suppose that $\beta \ge 13$. If $\beta \ge 20$, then we can reach a contradiction analogously to Claims \ref{radical_claim_1}, \ref{radical_claim_2} and \ref{radical_claim_3}. Now, suppose that $\beta \ \in \{ 13, 14, \ldots, 19\}$. In this case, a contradiction can be obtained via computer by showing that at least one of the following five statements is true for each such $\beta$ and any possible value of $t \bmod \beta$; see \cite{GitHub}.
    
    \bigskip\noindent
    \textbf{Statement 1:} There is an element of \eqref{tt_seq} with a unique remainder modulo $\beta$.

    If this is true, then Corollary \ref{cyclotomic_reduction_rad} and Lemma \ref{polynomial_division_lemma} imply that $\Phi_b(x)$ divides a polynomial of the form $c x^\alpha$ for some $c \in \mathbb{Z} \setminus \{ 0 \}$ and $\alpha \in \mathbb{N}_0$, which is impossible.

    \bigskip\noindent
    \textbf{Statement 2:} There are two elements of \eqref{tt_seq} that form an equivalence class modulo $\beta$, so that exactly one of them is from $\{ 6t + 15, 2t + 12 \}$ or exactly one of them is from $\{ 4t + 14, 4t + 13 \}$.

    In this case, Corollary \ref{cyclotomic_reduction_rad} and Lemma \ref{polynomial_division_lemma} imply that $\Phi_b(x)$ divides a polynomial of the form $c_1 x^{\alpha_1} + c_2 x^{\alpha_2}$, where $(|c_1|, |c_2|) \in \{ (3, 2), (3, 1), (2, 1) \}$ and $\alpha_1, \alpha_2 \in \mathbb{N}_0$. If we let $\zeta$ be a primitive $b$-th root of unity, then this means that some power of $\zeta$ equals $\pm \frac{3}{2}$ or $\pm 3$ or $\pm 2$, yielding a contradiction.

    \bigskip\noindent
    \textbf{Statement 3:} There are three elements of \eqref{tt_seq} that form an equivalence class modulo $\beta$, so that two of them are not from $\{ 6t + 15, 4t + 14, 4t + 13, 2t + 12 \}$, while the third is from $\{4t + 14, 4t + 13 \}$.

    Here, by Corollary \ref{cyclotomic_reduction_rad} and Lemma \ref{polynomial_division_lemma}, it follows that $\Phi_b(x)$ divides a polynomial of the form $c_1 x^{\alpha_1} + c_2 x^{\alpha_2} + c_3 x^{\alpha_3}$, where $|c_1| = |c_2| = 1$, $|c_3| = 3$ and $\alpha_1, \alpha_2, \alpha_3 \in \mathbb{N}_0$. Let $\zeta$ be a primitive $b$-th root of unity and note that
    \[
        c_1 \zeta^{\alpha_1 - \alpha_3} + c_2 \zeta^{\alpha_2 - \alpha_3} = -c_3 .
    \]
    The contradiction follows by observing that
    \[
        |c_1 \zeta^{\alpha_1 - \alpha_3} + c_2 \zeta^{\alpha_2 - \alpha_3}| \le |c_1 \zeta^{\alpha_1 - \alpha_3}| + |c_2 \zeta^{\alpha_2 - \alpha_3}| = 1 + 1 < 3 = |-c_3|.
    \]

    \bigskip\noindent
    \textbf{Statement 4:} $\beta = 13$ and the elements $4t + 20$ and $4t + 7$ form an equivalence class modulo $\beta$.

    In this case, Corollary \ref{cyclotomic_reduction_rad} and Lemma \ref{polynomial_division_lemma} give $\Phi_b(x) \mid -x^{4t + 20} + x^{4t + 7}$, i.e., $\Phi_b(x) \mid x^{13} - 1$. Hence, $b \mid 13$, which contradicts $\beta = 13$.

    \bigskip\noindent
    \textbf{Statement 5:} $\beta = 13$ and the elements $4t + 21$ and $4t + 8$ form an equivalence class modulo $\beta$.

    Here, Corollary \ref{cyclotomic_reduction_rad} and Lemma \ref{polynomial_division_lemma} give $\Phi_b(x) \mid -x^{4t + 21} - x^{4t + 8}$, i.e., $\Phi_b(x) \mid x^{13} + 1$. Therefore, $b \mid 26$, which contradicts $\beta = 13$.
\end{proof}

\begin{claim}\label{four_free_claim_4}
    Suppose that for some $t \in \mathbb{N}_0$ and $b \in \mathbb{N}$, we have $\Phi_b(x) \mid T_t(x)$. Then $2^2 \nmid b$.
\end{claim}
\begin{proof}
    By way of contradiction, suppose that $2^2 \mid b$. In this case, Lemma \ref{cyclotomic_reduction_prime} implies that the powers of all the nonzero terms of $\Phi_b(x)$ are even. Therefore, by Lemma \ref{polynomial_division_lemma}, the polynomial
    \begin{align*}
        x^{8t + 26} &+ x^{8t + 22} + x^{8t + 20} + x^{8t + 18} - x^{8t + 16} - x^{4t + 26} - x^{4t + 20} - x^{4t + 18}\\
            &+ 3 x^{4t + 14} + x^{4t + 10} - x^{4t + 8} + x^{4t + 6} - x^{4t + 4} + x^{4t + 2} - 2 x^{2t + 12} + x^{12} - x^{10} - x^2 - 1
    \end{align*}
    has a primitive $b$-th root of unity among its roots. This means that $S_t(x)$ has a root that is a primitive $\frac{b}{2}$-th root of unity, which cannot be possible due to Lemma \ref{st_lemma}.
\end{proof}

\noindent
The next claim can be proved analogously to Claim \ref{prime_claim_2}, so we omit its proof.

\begin{claim}\label{prime_claim_4}
    For any $t \in \mathbb{N}_0$ and prime $p \ge 41$, we have $\Phi_p(x) \nmid T_t(x)$ and $\Phi_{2p}(x) \nmid T_t(x)$.
\end{claim}

\noindent
We are now in a position to finalize the proof of Lemma \ref{tt_lemma}.

\begin{proof}[of Lemma \ref{tt_lemma}]
    By way of contradiction, suppose that $\Phi_b(x) \mid T_t(x)$ holds for some $t \in \mathbb{N}_0$ and $b \ge 3$. From Claims \ref{radical_claim_4} and \ref{four_free_claim_4}, we get $\frac{b}{\rad(b)} < 13$ and $2^2 \nmid b$. If $b$ has no odd prime factor below $41$, then we can apply Theorem \ref{filaschi_th} to repeatedly cancel out distinct prime factors of $b$ that are above $37$ until one such divisor is left. Therefore, $\Phi_p(x) \mid T_t(x)$ or $\Phi_{2p}(x) \mid T_t(x)$ holds for some prime $p \ge 41$, yielding a contradiction by virtue of Claim \ref{prime_claim_4}.

    Now, suppose that $b$ has an odd prime factor below $41$. By using Theorem \ref{filaschi_th}, we can cancel out all the prime factors of $b$ above $37$. With this in mind, $\Phi_{b'}(x) \mid T_t(x)$ holds for some $b' \ge 3$ whose prime factors are at most $37$ and such that $\frac{b'}{\rad(b')} < 13$ and $2^2 \nmid b'$. Besides, by Theorem \ref{filaschi_th}, we can assume without loss of generality that the distinct prime factors $p_1, p_2, \ldots, p_k$ of $b'$ satisfy $\sum_{j = 1}^k (p_j - 2) \le 36$. The proof can now be conveniently completed, e.g., by using a \texttt{SageMath} script, as shown in \cite{GitHub}.
\end{proof}

\section{Main result}\label{sc_main}

In this section, we finalize the proof of Theorem \ref{main_th}. Note that from Theorem \ref{base_vt_th} and $\mathfrak{N}_0^{\mathrm{reg}} = \mathfrak{N}_1^{\mathrm{reg}} = \mathfrak{N}_2^{\mathrm{reg}} = \varnothing$, it follows that $\mathfrak{N}_d^{\mathrm{VT}} = \mathfrak{N}_d^{\mathrm{Cay}} = \varnothing$ holds whenever $d$ is odd or $d < 4$. Besides, Theorem \ref{base_cay_th} determines $\mathfrak{N}_d^{\mathrm{VT}}$ and $\mathfrak{N}_d^{\mathrm{Cay}}$ for the case when $4 \mid d$ and $d \ge 4$. By virtue of Theorem~\ref{base_vt_th}, to complete the proof of Theorem~\ref{main_th}, it suffices to prove the existence of a $d$-regular Cayley nut graph of order $n$ for any parameters $d$ and $n$ such that $d \equiv 2 \pmod 4$, $d \ge 6$, $4 \mid n$ and $n \ge d + 6$. This can be accomplished by constructing Cayley nut graphs based on dihedral groups with the desired order and degree. We begin with the following two results.

\begin{proposition}\label{construction_1_prop}
    For any $t \in \mathbb{N}_0$ and even $m \ge 4t + 8$, the graph
    \begin{equation}\label{maux_1}
        \Cay(\Dih(m), \{ r^{\pm 1}, r^{ \pm 2}, r^{\pm 3}, \ldots, r^{\pm(2t + 1)} \} \cup \{ s, rs, r^4 s, r^6 s \} \cup \{ r^8 s, r^9 s, r^{10} s, \ldots, r^{4t + 7} s \})
    \end{equation}
    is an $(8t + 6)$-regular Cayley nut graph of order $2m$.
\end{proposition}
\begin{proof}
    Let
    \[
        A_\zeta = \begin{bmatrix}
            \sum_{j = 1}^{2t + 1} (\zeta^j + \zeta^{-j}) & 1 + \zeta^{-1} + \zeta^{-4} + \zeta^{-6} + \sum_{j = 8}^{4t + 7} \zeta^{-j}\\
            1 + \zeta + \zeta^4 + \zeta^6 + \sum_{j = 8}^{4t + 7} \zeta^j & \sum_{j = 1}^{2t + 1} (\zeta^j + \zeta^{-j})
        \end{bmatrix}
    \]
    for each $m$-th root of unity $\zeta$. Observe that
    \[
        A_1 = \begin{bmatrix}
            4t + 2 & 4t + 4\\
            4t + 4 & 4t + 2
        \end{bmatrix}
    \]
    is invertible, while
    \[
        A_{-1} = \begin{bmatrix}
            -2 & 2\\
            2 & -2
        \end{bmatrix}
    \]
    has a simple eigenvalue zero. Therefore, by Lemma \ref{dih_cay_lemma}, to complete the proof, it suffices to show that
    \begin{equation}\label{maux_2}
        \left( \sum_{j = 1}^{2t + 1} (\zeta^j + \zeta^{-j}) \right)^2 - \left( 1 + \zeta + \zeta^4 + \zeta^6 + \sum_{j = 8}^{4t + 7} \zeta^j \right) \left( 1 + \zeta^{-1} + \zeta^{-4} + \zeta^{-6} + \sum_{j = 8}^{4t + 7} \zeta^{-j} \right) = 0
    \end{equation}
    cannot hold for any $m$-th root of unity $\zeta \neq 1, - 1$.

    By way of contradiction, suppose that \eqref{maux_2} holds for some $m$-th root of unity $\zeta \neq 1, -1$. If we multiply both sides of \eqref{maux_2} by $(\zeta - 1)^2$, we get
    \begin{align}\label{maux_3}
    \begin{split}
        (\zeta^{2t + 2} &- \zeta + 1 - \zeta^{-2t-1})^2 - (\zeta^{4t + 8} - \zeta^8 + \zeta^7 - \zeta^6 + \zeta^5 - \zeta^4\\
        &+ \zeta^2 - 1)(\zeta - \zeta^{-1} + \zeta^{-3} - \zeta^{-4} + \zeta^{-5} - \zeta^{-6} + \zeta^{-7} - \zeta^{-4t-7}) = 0 .
    \end{split}
    \end{align}
    By expanding~\eqref{maux_3} and multiplying both sides by $\zeta^{4t + 7}$, it follows that
    \begin{align}\label{maux_4}
    \begin{split}
         - \zeta^{8t+16} &+ \zeta^{8t+14} - \zeta^{8t+12} + 2\zeta^{8t+11} - \zeta^{8t+10} + \zeta^{8t+9} - \zeta^{8t+8} - 2\zeta^{6t+10} + 2 \zeta^{6t+9} + \zeta^{4t+16}\\
         &- \zeta^{4t+15} + \zeta^{4t+12} - \zeta^{4t+11} + \zeta^{4t + 10} - 3\zeta^{4t+9} + 4\zeta^{4t+8} - 3\zeta^{4t+7} + \zeta^{4t+6} - \zeta^{4t+5}\\
         &+ \zeta^{4t+4} - \zeta^{4t+1} + \zeta^{4t} + 2\zeta^{2t+7} - 2 \zeta^{2t+6} -\zeta^8 + \zeta^7 - \zeta^6 + 2\zeta^5 - \zeta^4 + \zeta^2 - 1 = 0 .
    \end{split}
    \end{align}
    By factorizing \eqref{maux_4} accordingly, we obtain
    \begin{align*}
        (1 - \zeta)(\zeta^{8t + 15} &+ \zeta^{8t + 14} + \zeta^{8t + 11} - \zeta^{8t + 10} - \zeta^{8t + 8} + 2 \zeta^{6t + 9} - \zeta^{4t + 15} - \zeta^{4t + 11} - \zeta^{4t + 9}\\
        &+ 2 \zeta^{4t + 8} - 2 \zeta^{4t + 7} + \zeta^{4t + 6} + \zeta^{4t + 4} + \zeta^{4t} - 2 \zeta^{2t + 6} + \zeta^7 + \zeta^5 - \zeta^4 - \zeta - 1) = 0 .
    \end{align*}
    Since $\zeta \neq 1, -1$, the desired contradiction follows from Lemma \ref{rt_lemma}.
\end{proof}

\begin{proposition}\label{construction_2_prop}
    For any $t \in \mathbb{N}_0$ and even $m \ge 4t + 14$, the graph
    \begin{align*}
        \Cay(\Dih(m), &\{ r^{\pm 1}, r^{ \pm 2}, r^{\pm 3}, \ldots, r^{\pm(2t + 1)} \}\\
        & \cup \{ s, rs, r^2 s, r^5 s, r^7 s, r^9 s, r^{10} s  \} \cup \{ r^{13} s, r^{14} s, r^{15} s, \ldots, r^{4t + 13} s \})
    \end{align*}
    is an $(8t + 10)$-regular Cayley nut graph of order $2m$.
\end{proposition}
\begin{proof}
    Let
    \[
        A_\zeta = \begin{bmatrix}
            \sum_{j = 1}^{2t + 1} (\zeta^j + \zeta^{-j}) & \sum_{j \in \{ 0, 1, 2, 5, 7, 9, 10 \}} \zeta^{-j} + \sum_{j = 13}^{4t + 13} \zeta^{-j}\\
            \sum_{j \in \{ 0, 1, 2, 5, 7, 9, 10 \}} \zeta^j + \sum_{j = 13}^{4t + 13} \zeta^j & \sum_{j = 1}^{2t + 1} (\zeta^j + \zeta^{-j})
        \end{bmatrix}
    \]
    for each $m$-th root of unity $\zeta$. Observe that
    \[
        A_1 = \begin{bmatrix}
            4t + 2 & 4t + 8\\
            4t + 8 & 4t + 2
        \end{bmatrix}
    \]
    is invertible, while
    \[
        A_{-1} = \begin{bmatrix}
            -2 & -2\\
            -2 & -2
        \end{bmatrix}
    \]
    has a simple eigenvalue zero. By virtue of Lemma \ref{dih_cay_lemma}, to complete the proof, it is enough to show that
    \begin{equation}\label{maux_5}
        \left( \sum_{j = 1}^{2t + 1} (\zeta^j + \zeta^{-j}) \right)^2 - \left( \sum_{j \in \{ 0, 1, 2, 5, 7, 9, 10\}} \zeta^j + \sum_{j = 13}^{4t + 13} \zeta^j \right) \left( \sum_{j \in \{ 0, 1, 2, 5, 7, 9, 10\}} \zeta^{-j} + \sum_{j = 13}^{4t + 13} \zeta^{-j} \right) = 0
    \end{equation}
    does not hold for any $m$-th root of unity $\zeta \neq 1, -1$.

    By way of contradiction, suppose that \eqref{maux_5} holds for some $m$-th root of unity $\zeta \neq 1, -1$. By multiplying both sides of \eqref{maux_5} by $(\zeta - 1)^2$, we obtain
    \begin{align}\label{maux_6}
    \begin{split}
        (\zeta^{2t + 2} &- \zeta + 1 - \zeta^{-2t-1})^2 - (\zeta^{4t + 14} - \zeta^{13} + \zeta^{11} - \zeta^9 + \zeta^8 - \zeta^7 + \zeta^6 - \zeta^5\\
        &+ \zeta^3 - 1)(\zeta - \zeta^{-2} + \zeta^{-4} - \zeta^{-5} + \zeta^{-6} - \zeta^{-7} + \zeta^{-8} - \zeta^{-10} + \zeta^{-12} - \zeta^{-4t-13}) = 0 .
    \end{split}
    \end{align}
    If we expand \eqref{maux_6} and multiply both sides by $\zeta^{4t + 13}$, it follows that
    \begin{align}\label{maux_7}
    \begin{split}
        -\zeta^{8t + 28} &+ \zeta^{8t + 25} - \zeta^{8t + 23} + \zeta^{8t + 22} - \zeta^{8t + 21} + \zeta^{8t + 20} - \zeta^{8t + 19} + 2 \zeta^{8t + 17} - \zeta^{8t + 15}\\
        &- 2 \zeta^{6t + 16} + 2 \zeta^{6t + 15} + \zeta^{4t + 27} - \zeta^{4t + 25} - \zeta^{4t + 24} + \zeta^{4t + 23} + \zeta^{4t + 22} - 2 \zeta^{4t + 20}\\
        &+ 2 \zeta^{4t + 19} - \zeta^{4t + 17} - 3 \zeta^{4t + 15} + 6 \zeta^{4t + 14} - 3 \zeta^{4t + 13} - \zeta^{4t + 11} + 2 \zeta^{4t + 9} - 2 \zeta^{4t + 8}\\
        &+ \zeta^{4t + 6} + \zeta^{4t + 5} - \zeta^{4t + 4} - \zeta^{4t + 3} + \zeta^{4t + 1} + 2 \zeta^{2t + 13} - 2 \zeta^{2t + 12} - \zeta^{13} + 2 \zeta^{11}\\
        &- \zeta^9 + \zeta^8 - \zeta^7 + \zeta^6 - \zeta^5 + \zeta^3 - 1 = 0 .
    \end{split}
    \end{align}
    Now, by factorizing \eqref{maux_7}, we reach
    \begin{align*}
        (1 - \zeta)(\zeta^{8t + 27} &+ \zeta^{8t + 26} + \zeta^{8t + 25} + \zeta^{8t + 22} + \zeta^{8t + 20} + \zeta^{8t + 18} + \zeta^{8t + 17} - \zeta^{8t + 16} - \zeta^{8t + 15}\\
        &+ 2 \zeta^{6t + 15} - \zeta^{4t + 26} - \zeta^{4t + 25} + \zeta^{4t + 23} - \zeta^{4t + 21} - \zeta^{4t + 20} + \zeta^{4t + 19} - \zeta^{4t + 18}\\
        &- \zeta^{4t + 17} + 3 \zeta^{4t + 14} - 3 \zeta^{4t + 13} + \zeta^{4t + 10} + \zeta^{4t + 9} - \zeta^{4t + 8} + \zeta^{4t + 7} + \zeta^{4t + 6} - \zeta^{4t + 4}\\
        &+ \zeta^{4t + 2} + \zeta^{4t + 1} - 2 \zeta^{2t + 12} + \zeta^{12} + \zeta^{11} - \zeta^{10} - \zeta^9 - \zeta^7 - \zeta^5 - \zeta^2 - \zeta - 1) = 0 .
    \end{align*}
    Since $\zeta \neq 1, - 1$, a contradiction follows from Lemma \ref{tt_lemma}.
\end{proof}

\noindent
From Propositions \ref{construction_1_prop} and \ref{construction_2_prop} we obtain the next two corollaries, respectively.

\begin{corollary}
    Suppose that $d \ge 6$ is such that $d \equiv 6 \pmod 8$. Then for any $n \ge d + 10$ such that $4 \mid n$, there exists a $d$-regular Cayley nut graph of order $n$.
\end{corollary}

\begin{corollary}
    Suppose that $d \ge 6$ is such that $d \equiv 2 \pmod 8$. Then for any $n \ge d + 18$ such that $4 \mid n$, there exists a $d$-regular Cayley nut graph of order $n$.
\end{corollary}

Therefore, to complete the proof of Theorem \ref{main_th}, it remains to show the existence of a $d$-regular Cayley nut graph of order $n$ for the following two cases:
\begin{enumerate}[label=\textbf{(\roman*)}]
    \item $d \ge 6$, $d \equiv 6 \pmod 8$ and $n = d + 6$; and
    \item $d \ge 6$, $d \equiv 2 \pmod 8$ and $n \in \{ d + 6, d + 10, d + 14 \}$.
\end{enumerate}
We cover all but finitely many of the remaining $(n, d)$ pairs through the next three propositions.

\begin{proposition}\label{construction_3_prop}
    For any $d \ge 14$ such that $d \equiv 2 \pmod 4$, the graph
    \begin{equation}\label{maux_8}
        \overline{\Cay(\Dih(\tfrac{d + 6}{2}), \{ r^{\pm 2}, s, r^8 s, r^9 s \})}
    \end{equation}
    is a $d$-regular Cayley nut graph of order $d + 6$.
\end{proposition}
\begin{proof}
    By Lemmas \ref{reg_comp_lemma} and \ref{vt_nut_lemma}, it follows that the graph \eqref{maux_8} is a nut graph if and only if the graph
    \[
        \Cay(\Dih(\tfrac{d + 6}{2}), \{ r^{\pm 2}, s, r^8 s, r^9 s \})
    \]
    has a simple eigenvalue $-1$. Let $m = \frac{d + 6}{2}$ and for each $m$-th root of unity $\zeta$, let
    \[
        A_\zeta = \begin{bmatrix}
            1 + \zeta^2 + \zeta^{-2} & 1 + \zeta^{-8} + \zeta^{-9}\\
            1 + \zeta^8 + \zeta^9 & 1 + \zeta^2 + \zeta^{-2}
        \end{bmatrix} .
    \]
    Since the approach from Lemma \ref{dih_cay_lemma} can also be applied to graphs where loops are allowed, it suffices to prove that exactly one of the $A_\zeta$ matrices has a simple eigenvalue zero, while all the others are invertible.

    Since
    \[
        A_1 = \begin{bmatrix}
            3 & 3\\
            3 & 3
        \end{bmatrix}
    \]
    has a simple eigenvalue zero, it remains to verify that
    \begin{equation}\label{maux_9}
        (1 + \zeta^2 + \zeta^{-2})^2 - (1 + \zeta^8 + \zeta^9)(1 + \zeta^{-8} + \zeta^{-9}) = 0
    \end{equation}
    does not hold for any $m$-th root of unity $\zeta \neq 1$. By expanding \eqref{maux_9} and multiplying both sides by $\zeta^9$, we obtain
    \[
        -\zeta^{18} - \zeta^{17} + \zeta^{13} + 2 \zeta^{11} - \zeta^{10} - \zeta^8 + 2\zeta^7 + \zeta^5 - \zeta - 1 = 0 .
    \]
    The desired conclusion now follows by verifying that the polynomial
    \[
        -x^{18} - x^{17} + x^{13} + 2 x^{11} - x^{10} - x^8 + 2x^7 + x^5 - x - 1
    \]
    is not divisible by any cyclotomic polynomial $\Phi_b(x)$ with $b \ge 2$. This can easily be done by using a \texttt{SageMath} script; see \cite{GitHub}.
\end{proof}

\begin{proposition}\label{construction_4_prop}
    For any $d \ge 22$ such that $d \equiv 2 \pmod 4$, the graph
    \[
        \overline{\Cay(\Dih(\tfrac{d + 10}{2}), \{ r^{\pm 2}, r^{\pm 4}, s, r^2 s, r^6 s, r^7 s, r^{15} s \})}
    \]
    is a $d$-regular Cayley nut graph of order $d + 10$.
\end{proposition}
\begin{proof}
    Let $m = \frac{d + 10}{2}$ and for each $m$-th root of unity $\zeta$, let
    \[
        A_\zeta = \begin{bmatrix}
            1 + \zeta^2 + \zeta^{-2} + \zeta^4 + \zeta^{-4} & 1 + \zeta^{-2} + \zeta^{-6} + \zeta^{-7} + \zeta^{-15}\\
            1 + \zeta^2 + \zeta^6 + \zeta^7 + \zeta^{15} & 1 + \zeta^2 + \zeta^{-2} + \zeta^4 + \zeta^{-4}
        \end{bmatrix} .
    \]
    Analogously to Proposition \ref{construction_3_prop}, it is enough to prove that exactly one of the $A_\zeta$ matrices has a simple eigenvalue zero, while all the others are invertible. Note that
    \[
        A_1 = \begin{bmatrix}
            5 & 5\\
            5 & 5
        \end{bmatrix}
    \]
    has a simple eigenvalue zero. Therefore, it remains to show that
    \begin{equation}\label{maux_10}
        (1 + \zeta^2 + \zeta^{-2} + \zeta^4 + \zeta^{-4})^2 - (1 + \zeta^2 + \zeta^6 + \zeta^7 + \zeta^{15})(1 + \zeta^{-2} + \zeta^{-6} + \zeta^{-7} + \zeta^{-15}) = 0
    \end{equation}
    cannot hold for any $m$-th root of unity $\zeta \neq 1$. By expanding \eqref{maux_10} and multiplying both sides by $\zeta^{15}$, we get
    \[
        -\zeta^{30} - \zeta^{28} - \zeta^{24} - \zeta^{22} + \zeta^{21} - \zeta^{20} + 2 \zeta^{19} + 3 \zeta^{17} - \zeta^{16} - \zeta^{14} + 3 \zeta^{13} + 2 \zeta^{11} - \zeta^{10} + \zeta^9 - \zeta^8 - \zeta^6 - \zeta^2 - 1 = 0 .
    \]
    Analogously to Proposition \ref{construction_3_prop}, the result can be obtained by performing a computer-assisted verification via \texttt{SageMath}, as shown in \cite{GitHub}.
\end{proof}

\begin{proposition}\label{construction_5_prop}
    For any $d \ge 26$ such that $d \equiv 2 \pmod 4$, the graph
    \[
        \overline{\Cay(\Dih(\tfrac{d + 14}{2}), \{ r^{\pm 2}, r^{\pm 4}, r^{\pm 7}, s, r^2 s, r^6 s, r^7 s, r^{14} s, r^{17} s, r^{19} s \})}
    \]
    is a $d$-regular Cayley nut graph of order $d + 14$.
\end{proposition}
\begin{proof}
    Let $m = \frac{d + 14}{2}$ and for each $m$-th root of unity, let
    \[
        A_\zeta = \begin{bmatrix}
            1 + \zeta^2 + \zeta^{-2} + \zeta^4 + \zeta^{-4} + \zeta^7 + \zeta^{-7} & 1 + \zeta^{-2} + \zeta^{-6} + \zeta^{-7} + \zeta^{-14} + \zeta^{-17} + \zeta^{-19}\\
            1 + \zeta^2 + \zeta^6 + \zeta^7 + \zeta^{14} + \zeta^{17} + \zeta^{19} & 1 + \zeta^2 + \zeta^{-2} + \zeta^4 + \zeta^{-4} + \zeta^7 + \zeta^{-7}
        \end{bmatrix} .
    \]
    Analogously to Propositions \ref{construction_3_prop} and \ref{construction_4_prop}, it suffices to verify that exactly one of the $A_\zeta$ matrices has a simple eigenvalue zero, while all the others are invertible. Since
    \[
        A_1 = \begin{bmatrix}
            7 & 7\\
            7 & 7
        \end{bmatrix}
    \]
    has a simple eigenvalue zero, it remains to show that
    \begin{align}\label{maux_11}
    \begin{split}
        (1 &+ \zeta^2 + \zeta^{-2} + \zeta^4 + \zeta^{-4} + \zeta^7 + \zeta^{-7})^2 - (1 + \zeta^2 + \zeta^6 + \zeta^7\\
        &+ \zeta^{14} + \zeta^{17} + \zeta^{19})(1 + \zeta^{-2} + \zeta^{-6} + \zeta^{-7} + \zeta^{-14} + \zeta^{-17} + \zeta^{-19}) = 0
    \end{split}
    \end{align}
    does not hold for any $m$-th root of unity $\zeta \neq 1$. If we expand \eqref{maux_11} and multiply both sides by $\zeta^{19}$, we obtain
    \begin{align*}
        -\zeta^{38} &- 2 \zeta^{36}-\zeta^{34}-\zeta^{32}-2 \zeta^{31}+\zeta^{30}-\zeta^{29}+2 \zeta^{28}+\zeta^{25}+2 \zeta^{23}+\zeta^{22}+2 \zeta^{21}-\zeta^{20}\\
        &-\zeta^{18}+2 \zeta^{17}+\zeta^{16}+2 \zeta^{15}+\zeta^{13}+2 \zeta^{10}-\zeta^9+\zeta^8-2 \zeta^7-\zeta^6-\zeta^4-2 \zeta^2-1 = 0.
    \end{align*}
    Similarly to Propositions \ref{construction_3_prop} and \ref{construction_4_prop}, the proof can be completed by using a \texttt{SageMath} script; see \cite{GitHub}.
\end{proof}

With Propositions \ref{construction_3_prop}--\ref{construction_5_prop} in mind, it remains to verify the existence of a $d$-regular Cayley nut graph of order $n$ for each $(n, d) \in \{ (12, 6), (16, 10), (20, 10), (24, 10), (28, 18), (32, 18) \}$. This is not difficult to confirm through Tables \ref{table:vt} and \ref{table:cay}, which arise by performing an exhaustive search over all the vertex-transitive graphs of order below $48$; see \cite{HoRoy2020, RoyHo2020}. Besides, by using, e.g., \texttt{SageMath}, we can confirm that $\Cay(\Dih(6), \{ r^{\pm 1}, r^3, s, r^2 s, r^3 s \})$ is a $6$-regular Cayley nut graph of order $12$, while
\[
    \Cay(\Dih(m), \{ r^{\pm 1}, r^{\pm 2}, r^{\pm 3}, s, r^2 s, r^3 s, r^4 s \})
\]
is a $10$-regular Cayley nut graph of order $2m$ for each $m \in \{8, 10, 12 \}$ and
\[
    \Cay(\Dih(m), \{ r^{\pm 1}, r^{\pm 2}, r^{\pm 3}, r^{\pm 4}, r^{\pm 5}, s, r^2 s, r^3 s, r^4 s, r^5 s, r^6 s, r^7 s, r^8 s \})
\]
is an $18$-regular Cayley nut graph of order $2m$ for each $m \in \{14, 16 \}$. These observations complete the proof of Theorem \ref{main_th}.

We mention in passing that every vertex-transitive nut graph of order $8$, $12$, $14$, $22$, $38$ or $46$ is a Cayley graph since none of the numbers $8$, $12$, $14$, $22$, $38$ and $46$ is a non-Cayley number; see \cite{Marusic1983, Marusic1985, McKayPrae1994, McKayPrae1996} and the references therein. Although there exist non-Cayley vertex-transitive graphs of orders $10$, $28$ and $44$, none of them is a nut graph, hence the corresponding entries of Tables \ref{table:vt} and \ref{table:cay} are again the same.

The circulant graphs $\Cay(\mathbb{Z}_8, \{ \pm 1, \pm 2 \})$ and $\Cay(\mathbb{Z}_{10}, \{ \pm 1, \pm 2 \})$ are the unique $4$-regular vertex-transitive nut graph of order $8$ and $10$, respectively; see Figure \ref{vt_1}. Also, the noncirculant Cayley graphs
\[
    \Cay(\Dih(6), \{ r^{\pm 1}, r^3, s, r^2 s, r^3 s \}) \quad \mbox{and} \quad \overline{\Cay(\Dih(6), \{ r^{\pm 1}, s \})} \cong \overline{C_6 \Osq K_2}
\]
are the unique $6$- and $8$-regular vertex-transitive nut graph of order $12$, respectively; see Figure \ref{vt_2}. Moreover, there are exactly two $10$-regular vertex-transitive nut graphs of order $16$, one of which is the Cayley graph
\[
    \Cay(\Dih(8), \{ r^{\pm 1}, r^{\pm 2}, r^{\pm 3}, s, r^2 s, r^3 s, r^4 s \}) \cong \overline{\Cay(\Dih(8), \linebreak \{ r^4, r s, r^5 s, r^6 s, r^7 s\})},
\]
while the other is non-Cayley; see Figure \ref{vt_3}. Observe that the graph from Figure \ref{vt_3a} has a M\"{o}bius ladder \cite{Flapan1989} as a spanning subgraph, while the graph from Figure \ref{vt_3b} contains two disjoint M\"{o}bius ladders of order $8$.

It is not difficult to prove that the graph $\overline{\Cay(\mathbb{Z}_{2m}, \{ \pm 1, m \})}$, whose complement is a M\"{o}bius ladder, is a $(2m - 4)$-regular nut graph of order $2m$, for any $m \ge 4$ such that $4 \mid m$. Therefore, the graphs $\overline{\Cay(\mathbb{Z}_{16}, \{ \pm 1, 8 \})}$ and $\overline{\Cay(\mathbb{Z}_{32}, \{ \pm 1, 16 \})}$ are the unique $12$-regular vertex-transitive nut graph of order $16$ and $28$-regular vertex-transitive nut graph of order $32$, respectively, while $\overline{\Cay(\mathbb{Z}_{24}, \{ \pm 1, 12 \})}$ is one of the two $20$-regular vertex-transitive nut graphs of order $24$. The other $20$-regular vertex-transitive nut graph of order $24$ is also a Cayley graph and its complement is the Kronecker cover \cite{Pisanski2018} of the D\"{u}rer graph \cite{PiTu2002}; see Figure \ref{vt_4b}. Also, observe that there are exactly two $16$-regular vertex-transitive nut graphs of order $20$ and they are both Cayley graphs. Their complements are the prism graph $C_{10} \Osq K_2$ and the cubic hamiltonian graph that can be described as $[5, -5]^{10}$ using the exponential LCF notation \cite{Frucht1977}; see Figure \ref{vt_5}.

\begin{sidewaystable}[!p] 
\centering
\footnotesize
\setlength{\tabcolsep}{3.25pt}
\begin{tabular}{|c||rrrrrrrrrrrrrrrrrrrr|}
\hline 
\backslashbox{$d$}{$n$} & 8 & 10 & 12 & 14 & 16 & 18 & 20 & 22 & 24 & 26 & 28 & 30 & 32 & 34 & 36 & 38 & 40 & 42 & 44 & 46 \\
\hline 
\hline 
$4$ & $1$ & $1$ & $2$ & $2$ & $4$ & $2$ & $7$ & $4$ & $7$ & $5$ & $8$ & $6$ & $8$ & $7$ & $12$ & $8$ & $15$ & $9$ & $14$ & $10$\\
$6$ & $0$ & $0$ & $1$ & $0$ & $4$ & $0$ & $15$ & $0$ & $39$ & $0$ & $55$ & $0$ & $72$ & $0$ & $105$ & $0$ & $224$ & $0$ & $265$ & $0$\\
$8$ & --- & $0$ & $1$ & $3$ & $9$ & $14$ & $46$ & $38$ & $182$ & $92$ & $337$ & $180$ & $1152$ & $304$ & $1501$ & $476$ & $3344$ & $954$ & $3921$ & $1095$\\
$10$ & --- & --- & $0$ & $0$ & $2$ & $0$ & $30$ & $0$ & $231$ & $0$ & $520$ & $0$ & $1826$ & $0$ & $4170$ & $0$ & $11046$ & $0$ & $17496$ & $0$\\
$12$ & --- & --- & --- & $0$ & $1$ & $7$ & $41$ & $50$ & $337$ & $251$ & $1042$ & $1255$ & $6353$ & $2736$ & $12971$ & $7051$ & $52394$ & $23428$ & $76948$ & $34140$\\
$14$ & --- & --- & --- & --- & $0$ & $0$ & $9$ & $0$ & $201$ & $0$ & $1085$ & $0$ & $8284$ & $0$ & $27709$ & $0$ & $116986$ & $0$ & $225700$ & $0$\\
$16$ & --- & --- & --- & --- & --- & $0$ & $2$ & $9$ & $104$ & $147$ & $1134$ & $1293$ & $16569$ & $6607$ & $53738$ & $28713$ & $298855$ & $133196$ & $673180$ & $324689$\\
$18$ & --- & --- & --- & --- & --- & --- & $0$ & $0$ & $18$ & $0$ & $418$ & $0$ & $7178$ & $0$ & $40935$ & $0$ & $320178$ & $0$ & $909468$ & $0$\\
$20$ & --- & --- & --- & --- & --- & --- & --- & $0$ & $2$ & $13$ & $164$ & $389$ & $4650$ & $4192$ & $37632$ & $34901$ & $381977$ & $278017$ & $1277372$ & $1054239$\\
$22$ & --- & --- & --- & --- & --- & --- & --- & --- & $0$ & $0$ & $27$ & $0$ & $1232$ & $0$ & $21720$ & $0$ & $287618$ & $0$ & $1379665$ & $0$\\
$24$ & --- & --- & --- & --- & --- & --- & --- & --- & --- & $0$ & $3$ & $23$ & $417$ & $696$ & $10411$ & $12764$ & $222069$ & $211740$ & $1417958$ & $1244165$\\
$26$ & --- & --- & --- & --- & --- & --- & --- & --- & --- & --- & $0$ & $0$ & $28$ & $0$ & $2862$ & $0$ & $81731$ & $0$ & $781098$ & $0$\\
$28$ & --- & --- & --- & --- & --- & --- & --- & --- & --- & --- & --- & $0$ & $1$ & $23$ & $566$ & $1292$ & $31537$ & $43212$ & $388524$ & $537870$\\
$30$ & --- & --- & --- & --- & --- & --- & --- & --- & --- & --- & --- & --- & $0$ & $0$ & $55$ & $0$ & $5809$ & $0$ & $159876$ & $0$\\
$32$ & --- & --- & --- & --- & --- & --- & --- & --- & --- & --- & --- & --- & --- & $0$ & $4$ & $29$ & $1198$ & $2806$ & $53256$ & $81165$\\
$34$ & --- & --- & --- & --- & --- & --- & --- & --- & --- & --- & --- & --- & --- & --- & $0$ & $0$ & $80$ & $0$ & $9844$ & $0$\\
$36$ & --- & --- & --- & --- & --- & --- & --- & --- & --- & --- & --- & --- & --- & --- & --- & $0$ & $3$ & $54$ & $1378$ & $3790$\\
$38$ & --- & --- & --- & --- & --- & --- & --- & --- & --- & --- & --- & --- & --- & --- & --- & --- & $0$ & $0$ & $113$ & $0$\\
$40$ & --- & --- & --- & --- & --- & --- & --- & --- & --- & --- & --- & --- & --- & --- & --- & --- & --- & $0$ & $5$ & $43$\\
$42$ & --- & --- & --- & --- & --- & --- & --- & --- & --- & --- & --- & --- & --- & --- & --- & --- & --- & --- & $0$ & $0$\\
$44$ & --- & --- & --- & --- & --- & --- & --- & --- & --- & --- & --- & --- & --- & --- & --- & --- & --- & --- & --- & $0$\\
\hline 
\hline 
$\Sigma$ & $1$ & $1$ & $4$ & $5$ & $20$ & $23$ & $150$ & $101$ & $1121$ & $508$ & $4793$ & $3146$ & $47770$ & $14565$ & $214391$ & $85234$ & $1815064$ & $693416$ & $7376081$ & $3281206$ \\
\hline 
\end{tabular}
\caption{The number of vertex-transitive nut graphs of a given order $n$ and degree $d$.}
\label{table:vt}
\end{sidewaystable}

\begin{sidewaystable}[!p] 
\centering
\footnotesize
\setlength{\tabcolsep}{3.25pt}
\begin{tabular}{|c||rrrrrrrrrrrrrrrrrrrr|}
\hline 
\backslashbox{$d$}{$n$} & 8 & 10 & 12 & 14 & 16 & 18 & 20 & 22 & 24 & 26 & 28 & 30 & 32 & 34 & 36 & 38 & 40 & 42 & 44 & 46 \\
\hline 
\hline 
$4$ & $1$ & $1$ & $2$ & $2$ & $3$ & $2$ & $7$ & $4$ & $7$ & $5$ & $8$ & $5$ & $7$ & $7$ & $12$ & $8$ & $15$ & $9$ & $14$ & $10$\\
$6$ & $0$ & $0$ & $1$ & $0$ & $3$ & $0$ & $14$ & $0$ & $37$ & $0$ & $55$ & $0$ & $66$ & $0$ & $102$ & $0$ & $216$ & $0$ & $265$ & $0$\\
$8$ & --- & $0$ & $1$ & $3$ & $9$ & $13$ & $44$ & $38$ & $181$ & $86$ & $337$ & $179$ & $1090$ & $292$ & $1485$ & $476$ & $3303$ & $954$ & $3921$ & $1095$\\
$10$ & --- & --- & $0$ & $0$ & $1$ & $0$ & $29$ & $0$ & $230$ & $0$ & $520$ & $0$ & $1764$ & $0$ & $4142$ & $0$ & $10938$ & $0$ & $17496$ & $0$\\
$12$ & --- & --- & --- & $0$ & $1$ & $7$ & $40$ & $50$ & $337$ & $251$ & $1042$ & $1251$ & $6211$ & $2736$ & $12948$ & $7051$ & $52194$ & $23428$ & $76948$ & $34140$\\
$14$ & --- & --- & --- & --- & $0$ & $0$ & $8$ & $0$ & $200$ & $0$ & $1085$ & $0$ & $8068$ & $0$ & $27663$ & $0$ & $116598$ & $0$ & $225700$ & $0$\\
$16$ & --- & --- & --- & --- & --- & $0$ & $2$ & $9$ & $103$ & $141$ & $1134$ & $1289$ & $16198$ & $6557$ & $53667$ & $28713$ & $298282$ & $133196$ & $673180$ & $324689$\\
$18$ & --- & --- & --- & --- & --- & --- & $0$ & $0$ & $18$ & $0$ & $418$ & $0$ & $6974$ & $0$ & $40913$ & $0$ & $319663$ & $0$ & $909468$ & $0$\\
$20$ & --- & --- & --- & --- & --- & --- & --- & $0$ & $2$ & $13$ & $164$ & $386$ & $4519$ & $4192$ & $37578$ & $34901$ & $381495$ & $278017$ & $1277372$ & $1054239$\\
$22$ & --- & --- & --- & --- & --- & --- & --- & --- & $0$ & $0$ & $27$ & $0$ & $1192$ & $0$ & $21677$ & $0$ & $287139$ & $0$ & $1379665$ & $0$\\
$24$ & --- & --- & --- & --- & --- & --- & --- & --- & --- & $0$ & $3$ & $22$ & $371$ & $684$ & $10396$ & $12764$ & $221593$ & $211740$ & $1417958$ & $1244165$\\
$26$ & --- & --- & --- & --- & --- & --- & --- & --- & --- & --- & $0$ & $0$ & $24$ & $0$ & $2833$ & $0$ & $81421$ & $0$ & $781098$ & $0$\\
$28$ & --- & --- & --- & --- & --- & --- & --- & --- & --- & --- & --- & $0$ & $1$ & $23$ & $560$ & $1292$ & $31405$ & $43210$ & $388524$ & $537870$\\
$30$ & --- & --- & --- & --- & --- & --- & --- & --- & --- & --- & --- & --- & $0$ & $0$ & $52$ & $0$ & $5755$ & $0$ & $159876$ & $0$\\
$32$ & --- & --- & --- & --- & --- & --- & --- & --- & --- & --- & --- & --- & --- & $0$ & $4$ & $29$ & $1179$ & $2806$ & $53256$ & $81165$\\
$34$ & --- & --- & --- & --- & --- & --- & --- & --- & --- & --- & --- & --- & --- & --- & $0$ & $0$ & $77$ & $0$ & $9844$ & $0$\\
$36$ & --- & --- & --- & --- & --- & --- & --- & --- & --- & --- & --- & --- & --- & --- & --- & $0$ & $3$ & $54$ & $1378$ & $3790$\\
$38$ & --- & --- & --- & --- & --- & --- & --- & --- & --- & --- & --- & --- & --- & --- & --- & --- & $0$ & $0$ & $113$ & $0$\\
$40$ & --- & --- & --- & --- & --- & --- & --- & --- & --- & --- & --- & --- & --- & --- & --- & --- & --- & $0$ & $5$ & $43$\\
$42$ & --- & --- & --- & --- & --- & --- & --- & --- & --- & --- & --- & --- & --- & --- & --- & --- & --- & --- & $0$ & $0$\\
$44$ & --- & --- & --- & --- & --- & --- & --- & --- & --- & --- & --- & --- & --- & --- & --- & --- & --- & --- & --- & $0$\\
\hline
\hline
$\Sigma$ & $1$ & $1$ & $4$ & $5$ & $17$ & $22$ & $144$ & $101$ & $1115$ & $496$ & $4793$ & $3132$ & $46485$ & $14491$ & $214032$ & $85234$ & $1811276$ & $693414$ & $7376081$ & $3281206$\\
\hline 
\end{tabular}
\caption{The number of Cayley nut graphs of a given order $n$ and degree $d$.}
\label{table:cay}
\end{sidewaystable}

\begin{figure}[H]
\centering
\subcaptionbox{The unique $4$-regular vertex-transitive nut graph of order $8$, which is isomorphic to $\Cay(\mathbb{Z}_8, \{ \pm 1, \pm 2 \})$.}[0.47\textwidth]
{
    \centering
    \begin{tikzpicture}
        \tikzstyle{vertex}=[draw,circle,font=\scriptsize,minimum size=4pt,inner sep=1pt,fill=black]
        \tikzstyle{edge}=[draw,thick]

        \foreach \i in  {0,...,7} {
            \node[vertex] (v\i) at ({45*\i+22.5}:1.5) {};
        }
        \foreach \i in  {0,...,7} {
        	\pgfmathtruncatemacro{\j}{mod(\i + 1, 8)}
        	\pgfmathtruncatemacro{\k}{mod(\i + 2, 8)}
        	\path[edge] (v\i) -- (v\k);
        	\path[edge] (v\i) -- (v\j);
        }
    \end{tikzpicture}
}
\qquad
\subcaptionbox{The unique $4$-regular vertex-transitive nut graph of order $10$, which is isomorphic to $\Cay(\mathbb{Z}_{10}, \{ \pm 1, \pm 2 \})$.}[0.47\textwidth]
{
    \centering
    \begin{tikzpicture}
        \tikzstyle{vertex}=[draw,circle,font=\scriptsize,minimum size=4pt,inner sep=1pt,fill=black]
        \tikzstyle{edge}=[draw,thick]

        \foreach \i in  {0,...,9} {
            \node[vertex] (v\i) at ({36*\i}:1.5) {};
        }
        \foreach \i in  {0,...,9} {
        	\pgfmathtruncatemacro{\j}{mod(\i + 1, 10)}
        	\pgfmathtruncatemacro{\k}{mod(\i + 2, 10)}
        	\path[edge] (v\i) -- (v\k);
        	\path[edge] (v\i) -- (v\j);
        }
    \end{tikzpicture}
}
\caption{The unique $4$-regular vertex-transitive nut graph of order $8$ and $10$.}
\label{vt_1}
\end{figure}

\vspace{-1.4cm}

\begin{figure}[H]
\centering
\subcaptionbox{The unique $6$-regular vertex-transitive nut graph of \linebreak order $12$, which is isomorphic to $\Cay(\Dih(6), \{ r^{\pm 1}, \linebreak r^3, s, r^2 s, r^3 s \})$. The graph contains three cliques represented by shaded regions; edges within cliques are not drawn. Source:\ \cite[Figure~11]{BaFoPi2024}.}[0.47\textwidth]
{
    \centering
    \begin{tikzpicture}[scale=0.72]
        \definecolor{mygreen}{RGB}{19, 56, 190}
        \definecolor{sapphire}{RGB}{82, 178, 191}
        \draw[rotate=60,,color=sapphire,fill=sapphire!15!white] (0, 0) circle (1.5cm and 0.7cm);
        \draw[rotate around={-60:(-3, 0)},color=sapphire,fill=sapphire!15!white] (-3, 0) circle (1.5cm and 0.7cm);
        \draw[rotate around={0:(120:3)},color=sapphire,fill=sapphire!15!white] (120:3) circle (1.5cm and 0.7cm);
        \tikzstyle{vertex}=[draw,circle,font=\scriptsize,minimum size=6pt,inner sep=1pt,fill=mygreen]
        \tikzstyle{edge}=[draw,thick]
        \foreach \x/\cc/\fc in {0/black/black,1/black/black,2/black/black,3/black/black} {
             \node[vertex,fill=\cc] (u\x) at ($ (120:3) + (0.7 * \x, 0) - (1.05, 0)  $) {};
             \node[vertex,fill=\fc] (v\x) at ($ (0:0) + ({0.7 * cos(60) * \x}, {0.7 * sin(60) * \x}) - (60:1.05)  $) {};
             \node[vertex,fill=\cc] (w\x) at ($ (-3,0) + ({0.7 * cos(120) * \x}, {0.7 * sin(120) * \x}) - (120:1.05)  $) {};
        }
           \path[edge] (v0) -- (u2);
           \path[edge] (v1) -- (u3);
           \path[edge] (v2) -- (u0);
           \path[edge] (v2) -- (u2);
           \path[edge] (v3) -- (u1);
           \path[edge] (v3) -- (u3);
          \path[edge] (u3) -- (w2);
           \path[edge] (u2) -- (w3);
           \path[edge] (u1) -- (w0);
           \path[edge] (u1) -- (w2);
           \path[edge] (u0) -- (w1);
           \path[edge] (u0) -- (w3);
           \path[edge] (v3) -- (w1);
           \path[edge] (v2) -- (w0);
           \path[edge] (v1) -- (w3);
           \path[edge] (v1) -- (w1);
           \path[edge] (v0) -- (w2);
           \path[edge] (v0) -- (w0);
    \end{tikzpicture}
}
\qquad
\subcaptionbox{The complement of the unique $8$-regular vertex-transitive nut graph of order $12$, which is isomorphic to $\Cay(\Dih(6), \{ r^{\pm 1}, s \})$, i.e., the prism graph $C_6 \Osq K_2$. Source: \cite[Figure~1]{Damnjanovic2025_ADAM}.}[0.47\textwidth]
{
    \centering
    \begin{tikzpicture}[scale=0.9]
        \tikzstyle{vertex}=[draw,circle,font=\scriptsize,minimum size=4pt,inner sep=1pt,fill=black]
        \tikzstyle{edge}=[draw,thick]

        \foreach \i in  {0,...,5} {
            \node[vertex] (v\i) at ({60*\i}:1.2) {};
        }
        \foreach \i in  {6,...,11} {
            \node[vertex] (v\i) at ({60*\i}:2) {};
        }
        
        \foreach \i in  {0,...,5} {
        	\pgfmathtruncatemacro{\j}{mod(\i + 1, 6)};
        	\path[edge] (v\i) -- (v\j);
        }
        \foreach \i in  {6,...,11} {
        	\pgfmathtruncatemacro{\j}{6 + mod(\i + 1, 6)};
        	\path[edge] (v\i) -- (v\j);
        }
        \foreach \i in  {0,...,5} {
        	\pgfmathtruncatemacro{\j}{\i + 6};
        	\path[edge] (v\i) -- (v\j);
        }
    \end{tikzpicture}
}
\caption{The unique $6$- and $8$-regular vertex-transitive nut graph of order $12$ (drawn as the graph or its complement).}
\label{vt_2}
\end{figure}

\vspace{-1.2cm}

\begin{figure}[H]
\centering
\subcaptionbox{The complement of the unique $10$-regular Cay\-ley nut graph of order $16$, which is isomorphic to\linebreak $\Cay(\Dih(8), \{ r^4, r s, r^5 s, r^6 s, r^7 s\})$.\label{vt_3a}}[0.47\textwidth]
{
    \centering
    \begin{tikzpicture}[scale=0.9]
        \tikzstyle{vertex}=[draw,circle,font=\scriptsize,minimum size=4pt,inner sep=1pt,fill=black]
        \tikzstyle{edge}=[draw,thick]

        \foreach \i in  {0,2,4,6} {
            \node[vertex] (v\i) at ({45*\i-15+45}:1.5) {};
        }
        \foreach \i in  {1,3,5,7} {
            \node[vertex] (v\i) at ({45*\i-30+45}:1.5) {};
        }
        \foreach \i in  {8,10,12,14} {
            \node[vertex] (v\i) at ({45*\i-20+45}:2.25) {};
        }
        \foreach \i in  {9,11,13,15} {
            \node[vertex] (v\i) at ({45*\i-25+45}:2.25) {};
        }

        \foreach \i in  {0,2,4,6} {
        	\pgfmathtruncatemacro{\j}{8 + \i};
        	\pgfmathtruncatemacro{\k}{8 + mod(\i + 1, 8)};
            
        	\path[edge] (v\i) -- (v\j);
            \path[edge] (v\i) -- (v\k);
        }
        \foreach \i in  {1,3,5,7} {
        	\pgfmathtruncatemacro{\j}{8 + \i};
        	\pgfmathtruncatemacro{\k}{8 + mod(\i - 1, 8)};
            
        	\path[edge] (v\i) -- (v\j);
            \path[edge] (v\i) -- (v\k);
        }

        \foreach \i in  {0,2,4,6} {
        	\pgfmathtruncatemacro{\j}{mod(\i + 5, 8)};
            
        	\path[edge] (v\i) -- (v\j);
        }
        \foreach \i in  {1,3,5,7} {
        	\pgfmathtruncatemacro{\j}{mod(\i + 3, 8)};
            
        	\path[edge] (v\i) -- (v\j);
        }
        \foreach \i in  {8,10,12,14} {
        	\pgfmathtruncatemacro{\j}{8 + mod(\i + 5, 8)};
            
        	\path[edge] (v\i) -- (v\j);
        }
        \foreach \i in  {9,11,13,15} {
        	\pgfmathtruncatemacro{\j}{8 + mod(\i + 3, 8)};
            
        	\path[edge] (v\i) -- (v\j);
        }

        \foreach \i in  {0,1,2,3,4,6,7} {
        	\pgfmathtruncatemacro{\j}{mod(\i + 1, 8)};
            
        	\path[edge] (v\i) -- (v\j);
        }
        \path[edge] (v5) -- (v14);
        \foreach \i in  {8,9,10,11,12,14,15} {
        	\pgfmathtruncatemacro{\j}{8 + mod(\i + 1, 8)};
            
        	\path[edge] (v\i) -- (v\j);
        }
        \path[edge] (v13) -- (v6);
    \end{tikzpicture}
}
\qquad
\subcaptionbox{The complement of the unique $10$-regular non-Cayley vertex-transitive nut graph of order $16$.\label{vt_3b}}[0.47\textwidth]
{
    \centering
    \begin{tikzpicture}[scale=0.9]
        \tikzstyle{vertex}=[draw,circle,font=\scriptsize,minimum size=4pt,inner sep=1pt,fill=black]
        \tikzstyle{edge}=[draw,thick]

        \foreach \i in  {0,2,4,6} {
            \node[vertex] (v\i) at ({45*\i+30}:1.2) {};
        }
        \foreach \i in  {1,3,5,7} {
            \node[vertex] (v\i) at ({45*\i+15}:1.2) {};
        }
        \foreach \i in  {8,10,12,14} {
            \node[vertex] (v\i) at ({45*\i+35}:2.5) {};
        }
        \foreach \i in  {9,11,13,15} {
            \node[vertex] (v\i) at ({45*\i+10}:2.5) {};
        }

        \foreach \i in  {0,2,4,6} {
        	\pgfmathtruncatemacro{\j}{8 + \i};
        	\pgfmathtruncatemacro{\k}{8 + mod(\i + 1, 8)};
            
        	\path[edge] (v\i) -- (v\j);
            \path[edge] (v\i) -- (v\k);
        }
        \foreach \i in  {1,3,5,7} {
        	\pgfmathtruncatemacro{\j}{8 + \i};
        	\pgfmathtruncatemacro{\k}{8 + mod(\i - 1, 8)};
            
        	\path[edge] (v\i) -- (v\j);
            \path[edge] (v\i) -- (v\k);
        }

        \foreach \i in {8,10,14} {
        	\pgfmathtruncatemacro{\j}{8 + mod(\i + 1, 8)};
        	\pgfmathtruncatemacro{\k}{8 + mod(\i + 3, 8)};
            
        	\path[edge] (v\i) -- (v\j);
            \path[edge] (v\i) -- (v\k);
        }
        \path[edge] (v12) -- (v13);
        \path[edge] (v12) -- (v14);
        \foreach \i in {9,11,15} {
        	\pgfmathtruncatemacro{\j}{8 + mod(\i + 1, 8)};
            
        	\path[edge] (v\i) -- (v\j);
        }
        \path[edge] (v13) -- (v15);

        \path[edge] (v0) -- (v1);
        \path[edge] (v0) -- (v3);
        \path[edge] (v0) -- (v5);
        \path[edge] (v1) -- (v2);
        \path[edge] (v1) -- (v4);
        \path[edge] (v2) -- (v3);
        \path[edge] (v2) -- (v7);
        \path[edge] (v3) -- (v6);
        \path[edge] (v4) -- (v5);
        \path[edge] (v4) -- (v6);
        \path[edge] (v5) -- (v7);
        \path[edge] (v6) -- (v7);
    \end{tikzpicture}
}
\caption{The complements of the only two $10$-regular vertex-transitive nut graphs of order $16$.}
\label{vt_3}
\end{figure}

\vspace{-0.1cm}

\begin{figure}[H]
\centering
\subcaptionbox{The M\"{o}bius ladder of order $24$.}[0.47\textwidth]
{
    \centering
    \begin{tikzpicture}[scale=0.95]
        \tikzstyle{vertex}=[draw,circle,font=\scriptsize,minimum size=4pt,inner sep=1pt,fill=black]
        \tikzstyle{edge}=[draw,thick]

        \foreach \i in {0,...,23} {
            \node[vertex] (v\i) at ({15*\i}:1.8) {};
        }
        \foreach \i in {0,...,23} {
        	\pgfmathtruncatemacro{\j}{mod(\i + 1, 24)}
        	\path[edge] (v\i) -- (v\j);
        }
        \foreach \i in  {0,...,11} {
        	\pgfmathtruncatemacro{\j}{\i + 12};
        	\path[edge] (v\i) -- (v\j);
        }
    \end{tikzpicture}
}
\qquad
\subcaptionbox{The Kronecker cover of the D\"{u}rer graph.\label{vt_4b}}[0.47\textwidth]
{
    \centering
    \begin{tikzpicture}[scale=0.95]
        \tikzstyle{vertex}=[draw,circle,font=\scriptsize,minimum size=4pt,inner sep=1pt,fill=black]
        \tikzstyle{edge}=[draw,thick]

        \foreach \i in {0,...,7} {
            \node[vertex] (v\i) at ({45*\i+22.5}:1) {};
        }
        \foreach \i in {8,...,15} {
            \node[vertex] (v\i) at ({45*\i+22.5}:1.5) {};
        }
        \foreach \i in {16,...,23} {
            \node[vertex] (v\i) at ({45*\i+22.5}:2) {};
        }

        \foreach \i in {0,...,7} {
        	\pgfmathtruncatemacro{\j}{mod(\i + 1, 8)};
        	\path[edge] (v\i) -- (v\j);
        }
        \foreach \i in {16,...,23} {
        	\pgfmathtruncatemacro{\j}{16 + mod(\i + 1, 8)};
        	\path[edge] (v\i) -- (v\j);
        }
        \foreach \i in {0,...,15} {
        	\pgfmathtruncatemacro{\j}{\i + 8};
        	\path[edge] (v\i) -- (v\j);
        }
        \foreach \i in {8, 10, 12, 14} {
        	\pgfmathtruncatemacro{\j}{8 + mod(\i + 3, 8)};
        	\path[edge] (v\i) -- (v\j);
        }
    \end{tikzpicture}
}
\caption{The complements of the only two $20$-regular vertex-transitive nut graphs of order $24$, each of which is a Cayley graph.}
\end{figure}

\begin{figure}[H]
\centering
\subcaptionbox{The prism graph $C_{10} \Osq K_2$.}[0.47\textwidth]
{
    \centering
    \begin{tikzpicture}[scale=0.92]
        \tikzstyle{vertex}=[draw,circle,font=\scriptsize,minimum size=4pt,inner sep=1pt,fill=black]
        \tikzstyle{edge}=[draw,thick]

        \foreach \i in {0,...,9} {
            \node[vertex] (v\i) at ({36*\i}:1) {};
        }
        \foreach \i in {10,...,19} {
            \node[vertex] (v\i) at ({36*\i}:2) {};
        }
        
        \foreach \i in {0,...,9} {
        	\pgfmathtruncatemacro{\j}{mod(\i + 1, 10)}
            \pgfmathtruncatemacro{\k}{\i + 10}
        	\path[edge] (v\i) -- (v\j);
        	\path[edge] (v\i) -- (v\k);
        }
        \foreach \i in  {10,...,19} {
        	\pgfmathtruncatemacro{\j}{10 + mod(\i + 1, 10)}
        	\path[edge] (v\i) -- (v\j);
        }
    \end{tikzpicture}
}
\qquad
\subcaptionbox{The cubic hamiltonian graph with the exponential LCF notation $[5, -5]^{10}$.}[0.47\textwidth]
{
    \centering
    \begin{tikzpicture}[scale=0.9]
        \tikzstyle{vertex}=[draw,circle,font=\scriptsize,minimum size=4pt,inner sep=1pt,fill=black]
        \tikzstyle{edge}=[draw,thick]

        \foreach \i in {0,...,19} {
            \node[vertex] (v\i) at ({18*\i+9}:2) {};
        }

        \foreach \i in {0,...,19} {
        	\pgfmathtruncatemacro{\j}{mod(\i + 1, 20)};
        	\path[edge] (v\i) -- (v\j);
        }
        \foreach \i in {0,2,4,6,8,10,12,14,16,18} {
        	\pgfmathtruncatemacro{\j}{mod(\i + 5, 20)};
        	\path[edge] (v\i) -- (v\j);
        }
    \end{tikzpicture}
}
\caption{The complements of the only two $16$-regular vertex-transitive nut graphs of order $20$, each of which is a Cayley graph.}
\label{vt_5}
\end{figure}

\section{Conclusion}\label{sc_conclusion}

Theorem \ref{main_th} completely resolves the vertex-transitive (resp.\ Cayley) nut graph order--degree existence problem, thus providing the solution to Problem \ref{vt_problem} and an inverse result for Theorem \ref{base_vt_th}. Its results can be alternatively stated as follows.

\begin{corollary}
    For any $n \in \mathbb{N}$ and $d \in \mathbb{N}_0$, there exists a $d$-regular vertex-transitive nut graph of order $n$ if and only if:
    \begin{enumerate}[label=\textbf{(\roman*)}]
        \item $n$ and $d$ are both even, with at least one of them divisible by four; and
        \item $d \ge 4$ and $n \ge d + 4$.
    \end{enumerate}
\end{corollary}

\begin{corollary}
    For any $n \in \mathbb{N}$ and $d \in \mathbb{N}_0$, there exists a $d$-regular Cayley nut graph of order $n$ if and only if:
    \begin{enumerate}[label=\textbf{(\roman*)}]
        \item $n$ and $d$ are both even, with at least one of them divisible by four; and
        \item $d \ge 4$ and $n \ge d + 4$.
    \end{enumerate}
\end{corollary}

All the constructions used in Section \ref{sc_main} relied on Cayley graphs based on dihedral groups. In \cite{Damnjanovic2025_ADAM}, it was shown that for any $d \ge 8$ such that $8 \mid d$, the graph $\overline{C_\frac{d + 4}{2} \Osq K_2} \cong \overline{\Cay(\Dih(\frac{d + 4}{2}), \{ r^{\pm 1}, s \})}$ is a $d$-regular Cayley nut graph of order $d + 4$. Besides, it is not difficult to verify by using, e.g., \texttt{SageMath}, that $\Cay(\Dih(8), \{ r^{\pm 1}, r^{\pm 2}, r^{\pm 3}, s, r^2 s \})$ is an $8$-regular Cayley nut graph of order $16$. With all of this in mind together with Theorem \ref{base_circ_th}, we reach the next result.

\begin{theorem}\label{suff_th}
    Suppose that $n \in \mathbb{N}$ and $d \in \mathbb{N}_0$ are such that:
    \begin{enumerate}[label=\textbf{(\roman*)}]
        \item $n$ and $d$ are both even, with at least one of them divisible by four; and
        \item $d \ge 4$ and $n \ge d + 4$.
    \end{enumerate}
    Then there is a $d$-regular Cayley nut graph of order $n$ for the cyclic or dihedral group.
\end{theorem}

\noindent
From here, we obtain the following conclusion.

\begin{corollary}
The cyclic and dihedral groups are sufficient to construct Cayley nut graphs that cover all the possible combinations of orders and degrees attainable by a vertex-transitive nut graph.
\end{corollary}

A \emph{bicirculant graph} is a graph that has an automorphism with two orbits of equal size. These graphs are the derived graphs of $\mathbb{Z}_m$-voltage pregraphs of order two; see \cite[Section~3.5]{PiSe2013} and \cite{MaMaPo2004, MaNeSko2000, Pisanski2007, PoTo2020}. As shown in Lemma~\ref{dih_cay_lemma}, the Cayley graphs for dihedral groups are a subclass of the bicirculant graphs. Therefore, it is natural to extend the investigation of the nut property to bicirculant graphs. To this end, we need the next proposition.

\begin{proposition}\label{bicirc_prop}
    For any $d$-regular bicirculant nut graph of order $n$, the following holds:
    \begin{enumerate}[label=\textbf{(\roman*)}]
        \item $n$ and $d$ are both even, with at least one of them divisible by four; and
        \item $d \ge 4$ and $n \ge d + 4$.
    \end{enumerate}
\end{proposition}
\begin{proof}
    Let $G$ be a $d$-regular bicirculant nut graph of order $n$. Observe that $A(G)$ has the form
    \[
        \begin{bmatrix}
            C_0 & C_1^\intercal \\
            C_1 & C_2\\
        \end{bmatrix},
    \]
    where $C_1$ is a binary circulant matrix, while $C_0$ and $C_2$ are binary symmetric circulant matrices with zero diagonal. Let $n = 2m$, so that $C_0, C_1, C_2 \in \mathbb{R}^{m \times m}$, and let $S_0$, $S_1$ and $S_2$ be the connection sets of $C_0$, $C_1$ and $C_2$, respectively. Now, for each $m$-th root of unity $\zeta$, let
    \[
        A_\zeta = \begin{bmatrix}
            \sum_{j \in S_0} \zeta^j & \sum_{j \in S_1} \zeta^{-j}\\
            \sum_{j \in S_1} \zeta^j & \sum_{j \in S_2} \zeta^j
        \end{bmatrix} .
    \]
    By arguing analogously to Lemma \ref{dih_cay_lemma}, it follows that exactly one of the $A_\zeta$ matrices has a simple eigenvalue zero, while all the others are invertible.
    
    Let $\zeta_0$ be the unique $m$-th root of unity such that $A_{\zeta_0}$ has an eigenvalue zero. We trivially observe that $\zeta_0$ must be real, since otherwise both $A_{\zeta_0}$ and $A_{\overline{\zeta_0}}$ would have an eigenvalue zero. Note that $d = |S_0| + |S_1| = |S_2| + |S_1|$ and
    \[
        A_1 = \begin{bmatrix}
            |S_0| & |S_1|\\
            |S_1| & |S_0|
        \end{bmatrix} .
    \]
    Regardless of whether $\zeta_0 = 1$ or $\zeta_0 = -1$, it is not difficult to see that $|S_0|$ and $|S_1|$ are of the same parity, which means that $d$ is even. Since $\mathfrak{N}_0^{\mathrm{reg}} = \mathfrak{N}_2^{\mathrm{reg}} = \varnothing$, we get $d \ge 4$. Also, the only $d$-regular graph of order $d + 2$ is $\overline{\frac{d + 2}{2} K_2}$, hence
    \[
        \eta\left( \overline{\tfrac{d + 2}{2} K_2} \right) = \tfrac{d + 2}{2} > 1
    \]
    implies that $n \ge d + 4$.

    It remains to prove that $4 \mid n$ or $4 \mid d$. By way of contradiction, suppose that $m$ is odd and $d \equiv 2 \pmod 4$. In this case, $-1$ is not an $m$-th root of unity, hence $\zeta_0 = 1$. Since $A_1$ is noninvertible, we have $|S_0| = |S_1|$, which implies that $|S_0|$ is odd. This yields a contradiction because a (cyclic) group of odd order has no self-inverse element apart from the identity.
\end{proof}

\noindent
As an immediate corollary to Theorem \ref{suff_th} and Proposition \ref{bicirc_prop}, we obtain the following result.

\begin{corollary}\label{bicirc_cor}
    For any $n \in \mathbb{N}$ and $d \in \mathbb{N}_0$, there exists a $d$-regular nut graph of order $n$ that is a circulant or bicirculant graph if and only if:
    \begin{enumerate}[label=\textbf{(\roman*)}]
        \item $n$ and $d$ are both even, with at least one of them divisible by four; and
        \item $d \ge 4$ and $n \ge d + 4$.
    \end{enumerate}
\end{corollary}

Note that bicirculant nut graphs need not be regular. For example, the graph with the adjacency matrix
\[
    \begin{bmatrix}
        C_0 & C_1^\intercal\\
        C_1 & C_2
    \end{bmatrix} ,
\]
with $C_0, C_1, C_2 \in \mathbb{R}^{18}$ having the connection sets
\[
    \{ 1 \}, \quad \{0, 2\} \quad \mbox{and} \quad \{ 1, 2, 3 \},
\]
respectively, is a bicirculant nut graph of order $36$ where the vertices from one orbit are of degree four, while the vertices from the other orbit are of degree eight. Let $\mathfrak{N}_{d_1, d_2}^{\mathrm{bicirc}}$ be the set of all the orders attainable by a bicirculant nut graph where the vertices from the two orbits have degrees $d_1$ and $d_2$, respectively. It is natural to pose the following problem.

\begin{problem}
    For any $d_1, d_2 \in \mathbb{N}_0$, determine the set $\mathfrak{N}_{d_1, d_2}^{\mathrm{bicirc}}$.
\end{problem}

We end the paper with two more corollaries of Theorem \ref{main_th}.

\begin{corollary}\label{main_cor_1}
    For any $d \ge 4$ such that $4 \mid d$, we have
    \[
        \mathfrak{N}_d^{\mathrm{reg}} \supseteq \{ n \in \mathbb{N} : \mbox{$n$ is even and } n \ge d + 4 \} .
    \]
\end{corollary}

\begin{corollary}\label{main_cor_2}
    For any $d \ge 6$ such that $d \equiv 2 \pmod 4$, we have
    \[
        \mathfrak{N}_d^{\mathrm{reg}} \supseteq \{ n \in \mathbb{N} : \mbox{$4 \mid n$ and } n \ge d + 6 \} .
    \]
\end{corollary}

\noindent
Although Corollaries \ref{main_cor_1} and \ref{main_cor_2} give a partial solution to Problem \ref{reg_problem} for the case when $d$ is even, the regular nut graph order--degree existence problem seems much more difficult to solve. Corollary \ref{bicirc_cor} justifies this claim and implies that different constructions not relying on circulant or bicirculant graphs would need to be used to further investigate Problem \ref{reg_problem}.

\acknowledgements
\label{sec:ack}
The author is grateful to Nino Bašić for all of his useful comments and his overall help with the data processing concerning Tables \ref{table:vt} and \ref{table:cay}.

\nocite{*}
\bibliographystyle{abbrvnat}
\bibliography{bibliography}
\label{sec:biblio}

\end{document}